\documentclass[11pt, reqno]{amsart} 
\usepackage{amssymb,amscd,amsfonts,amsbsy}
\usepackage{latexsym}
\usepackage{exscale}
\usepackage{amsmath,amsthm,amsfonts}
\usepackage{mathrsfs}
\usepackage{xcolor} 
\usepackage[colorlinks=true,linkcolor=blue,citecolor=red,urlcolor=red, backref=page]{hyperref} 
\usepackage{esint} 
\usepackage{stmaryrd}
\usepackage{pifont}

\usepackage[utf8]{inputenc}

\calclayout

\allowdisplaybreaks

\newtheorem{theorem}{Theorem}[section]
\newtheorem{proposition}[theorem]{Proposition}
\newtheorem{lemma}[theorem]{Lemma}

\newtheorem{definition}[theorem]{Definition}

\newtheorem{remark}[theorem]{Remark}

\numberwithin{equation}{section} 

\let\al=\alpha
\let\b=\beta
\let\g=\gamma

\let\z=\zeta
\let\la=\lambda
\let\r=\rho
\let\s=\sigma

\let\om=\omega
\let\G= \Gamma

\let\La=\Lambda

\let\P=\Phi

\let\ep=\epsilon
\let\va=\varphi

\let\fy=\infty

\def\scrP{\mathscr{P}}

\def\calA{\mathcal {A}}

\def\calD{\mathcal {D}}

\def\calI{\mathcal {I}}
\def\calM{\mathcal {M}}
\def\calT{\mathcal {T}}
\def\calS{\mathcal {S}}

\def\scrF{\mathscr{F}}

\newcommand{\be}{\begin{equation*}}
\newcommand{\ee}{\end{equation*}}
\newcommand{\ben}{\begin{equation}}
\newcommand{\een}{\end{equation}}
\newcommand{\bn}{\begin{enumerate}}
\newcommand{\en}{\end{enumerate}}

\def\rr{{\mathbb R}}
\def\rd{{{\rr}^d}}
\def\rdd{{{\rr}^{2d}}}
\def\rddd{{{\rr}^{4d}}}
\def\zd{{{\mathbb{Z}}^d}}
\def\zdd{{{\mathbb{Z}}^{2d}}}
\def\zddd{{{\mathbb{Z}}^{4d}}}

\def\-fl1{\scrF^{-1}L^1}
\def\fl1{\scrF L^1}

\begin{document}
\title[Boundedness on weighted modulation spaces
of $\tau$-Wigner distributions]
{Characterization of boundedness on weighted modulation spaces
of $\tau$-Wigner distributions}
\author{WEICHAO GUO}
\address{School of Science, Jimei University, Xiamen, 361021, P.R.China}
\email{weichaoguomath@gmail.com}
\author{jiecheng chen}
\address{Department of Mathematics, Zhejiang Normal University,
Jinhua, 321004, P.R.China} \email{jcchen@zjnu.edu.cn}
\author{DASHAN FAN}
\address{Department of Mathematics, University of Wisconsin-Milwaukee, Milwaukee, WI 53201, USA}
\email{fan@uwm.edu}
\author{GUOPING ZHAO}
\address{School of Applied Mathematics, Xiamen University of Technology,
Xiamen, 361024, P.R.China} \email{guopingzhaomath@gmail.com}

\begin{abstract}
This paper is devoted to give several characterizations on a more general level
for the boundedness of $\tau$-Wigner distributions acting from weighted modulation spaces to
weighted modulation and Wiener amalgam spaces.
As applications, sharp exponents are obtained for the boundedness of $\tau$-Wigner distributions
on modulation spaces with power weights. We also recapture the main theorems of Wigner distribution obtained in \cite{CorderoNicola2018IMRNI,Cordero2020a}. As consequences, the characterizations of the boundedness on weighted modulation spaces
of several types of pseudodifferential operators are established. In particular, we give the sharp exponents for the boundedness
of pseudodifferential operators with symbols in Sj\"{o}strand's class and the corresponding Wiener amalgam spaces.
\end{abstract}

\maketitle


\section{INTRODUCTION}
The study of cross-Wigner distribution has a long history. It was first introduced in 1932
in E.Wigner's ground-breaking paper \cite{Wigner1932PR}, and then introduced in 1948 by J.Ville \cite{Ville1948CeT} in the field of signal analysis.
Let us recall the definition as follows.

Given two functions $f_1,f_2\in L^2(\rd)$, the cross-Wigner distribution $W(f_1,f_2)$ is defined by
\be
W(f_1,f_2)(x,\xi):=\int_{\rd}f_1(x+\frac{t}{2})\overline{f_2(x-\frac{t}{2})}e^{-2\pi it\cdot\xi}dt.
\ee
Let $\calT_s$ be the symmetric coordinate change defined by
\be
\calT_sF(x,t)=F(x+\frac{t}{2},x-\frac{t}{2}),
\ee
and let $\scrF_2$ be the partial Fourier transform in the second variable defined by
\be
\scrF_2F(x,\xi)=\int_{\rd}F(x,t)e^{-2\pi it\cdot\xi}dt.
\ee
The cross-Wigner distribution can be written as
\be
W(f_1,f_2)=\scrF_2\calT_s(f_1\otimes \bar{f_2}).
\ee
For $f=f_1=f_2$, $Wf=W(f,f)$ is simply called the Wigner distribution of $f$.
For simplicity, in the remaining part of this paper, we omit the word ``cross'' no matter whether $f_1=f_2$ or not.

As an important time-frequency representation, the Wigner distribution is closed related to the short-time Fourier transform (STFT)
defined by
\be
V_gf(x,\xi):=\int_{\rd}f(t)\overline{g(t-x)}e^{-2\pi it\cdot \xi}dt,\ \ \  f,g\in L^2(\rd).
\ee
In fact, a direct calculation shows that
\be
W(f,g)(x,\xi)=2^de^{4\pi i x\cdot \xi}V_{\calI g}f(2x,2\xi),\ \ \  f,g\in L^2(\rd),
\ee
where $\calI g(t)=g(-t)$.

On the other hand, the pseudodifferential operator in the Weyl form, i.e., the Weyl operator $L_{\sigma}$ with symbol $\s\in \calS'(\rdd)$
can be defined by means of duality pairing between the symbol and the Wigner distribution:
\be
\langle L_{\sigma}f, g\rangle=\langle \sigma, W(g,f)\rangle,\ \ \ f,g\in \calS(\rd).
\ee

The localization operator $A_{a}^{\va_1,\va_2}$  with symbol $a\in \calS'(\rdd)$,
analysis window $\va_1\in \calS(\rd)$, and synthesis window function $\va_2\in \calS(\rd)$
can be regarded as the Weyl operator whose symbol is the convolution of $a$
with the Wigner distribution of the windows $\va_1$ and $\va_2$:
\be
A_{a}^{\va_1,\va_2}=L_{a\ast W(\va_2,\va_1)}.
\ee

For $\tau\in [0,1]$, a more general time-frequency representation, namely, the cross-$\tau$-Wigner distribution
of $f_1,f_2\in L^2(\rd)$ is defined by
\be
W_{\tau}(f_1,f_2)(x,\xi):=\int_{\rd}f_1(x+\tau t)\overline{f_2(x-(1-\tau)t)}e^{-2\pi i\xi\cdot t}dt.
\ee
For $f=f_1=f_2$, $W_{\tau}f:=W_{\tau}(f,f)$ is simply called the $\tau$-Wigner distribution of $f$.
For simplicity, we omit the word ``cross'' in the remaining part of this paper.

Note that $\tau$-Wigner distribution is a generalization of the Wigner distribution.
Varying the parameter $\tau$, $W_{\tau}(f_1,f_2)$ is a family of time-frequency representations:

For $\tau=1/2$, $W_{1/2}(f_1,f_2)$ becomes the Wigner distribution $W(f_1,f_2)$.

For $\tau=0$, $W_0(f_1,f_2)$ coincides with the Rihaczek distribution $R(f_1,f_2)$:
\be
W_0(f_1,f_2)(x,\xi)=R(f_1,f_2)(x,\xi)=e^{-2\pi ix\cdot\xi}f_1(x)\overline{\hat{f_2}(\xi)}.
\ee

For $\tau=1$, $W_1(f_1,f_2)$ coincides with the conjugate Rihaczek distribution $R^*(f_1,f_2)$:
\be
W_1(f_1,f_2)(x,\xi)=R^*(f_1,f_2)(x,\xi)=\overline{R(f_2,f_1)(x,\xi)}
=e^{2\pi ix\cdot\xi}\overline{f_2(x)}\hat{f_1}(\xi).
\ee

By the Rihaczek distribution, the famous Kohn-Nirenberg operator $K_{\s}$ with symbol $\s$ can be defined weakly as:
\be
\langle K_{\s}f, g\rangle=\langle \s, R(g,f)\rangle,\ \ \ f,g\in \calS(\rd),\ \s\in \calS'(\rdd).
\ee
In general, for $\tau\in [0,1]$,
the so-called $\tau$-operators or Shubin operators \cite{Shubin2001} can be defined as
\be
\langle \text{OP}_{\tau}(\s)f, g\rangle=\langle \s, W_{\tau}(g,f)\rangle,\ \ \ f,g\in \calS(\rd).
\ee
Note that $\text{OP}_0(\s)$ coincides with the Khon-Nirenberg operator $K_{\s}$, $\text{OP}_{1/2}(\s)$
is just the Weyl operator $L_{\s}$. As the adjoint operator of $\text{OP}_0(\s)$, $\text{OP}_1(\s)$ is also called anti-Kohn-Nirenberg operator.

According to the above relations, the boundedness of several important operators
has direct connections with the corresponding boundedness of $\tau$-Wigner distributions.
Hence, it is important to establish the boundedness results of $\tau$-Wigner distribution on function spaces.
Among them, the boundedness acting on modulation spaces has its important position, since it
has closed relationship with time-frequency analysis.

Modulation spaces were invented by H. Feichtinger \cite{Feichtinger1983TRUoV} in 1983.
Nowadays, they have been fully recognized as the ``right'' function spaces
for time-frequency analysis. More precisely, modulation spaces
are defined by measuring the decay and integrability of the STFT as following:
\be
M^{p,q}_m(\rd)=\{f\in \calS'(\rd): V_gf\in L^{p,q}_m(\rdd) \}
\ee
endowed with the obvious (quasi-)norm, where $L^{p,q}_m(\rdd)$ are weighted mixed-norm Lebesgue spaces with the weight $m$,
see Section 2 for more details. We use $\calM^{p,q}_m(\rd)$ to denote the $\calS(\rd)$ closure in $M^{p,q}_m(\rd)$.

For the power weights
\be
v_{s,t}(z)=\langle z_1\rangle^s\langle z_2\rangle^t,\ \ \ v_s(z)=\langle z\rangle^s=(1+|z|^2)^{s/2} \ \ s,t\in \rr,
\ \ z=(z_1,z_2)\in \rdd,
\ee
the problem for \emph{the boundedness of $\tau$-Wigner distribution acting from weighted modulation spaces to
weighted modulation spaces} (BMM)
is to find the full range of exponents of $p_i, q_i, p, q \in (0,\fy]$, $s_i, t_i\in \rr$, $i=1,2$
for the boundedness:
\be
W_{\tau}: \calM^{p_1,q_1}_{v_{s_1,t_1}}(\rd)\times \calM^{p_2,q_2}_{v_{s_2,t_2}}(\rd)\longrightarrow M^{p,q}_{1\otimes v_{s}}(\rdd),
\ee
that is,
\be
\|W_{\tau}(f_1,f_2)\|_{M^{p,q}_{1\otimes v_{s}}}\lesssim \|f_1\|_{M^{p_1,q_1}_{v_{s_1,t_1}}}\cdot \|f_2\|_{M^{p_2,q_2}_{v_{s_2,t_2}}},\ \ \ f,g\in \calS(\rd),
\ee
where we write $(1\otimes v_s)(z,\z):= v_s(\z)$ for $(z,\z)\in \rddd$.

Note that, to avoid the fact that $\calS(\rd)$ is not dense in some endpoint spaces, such as $M^{p,q}_m$ with $p=\fy$ or $q=\fy$,
we only consider the action of $\tau$-Wigner distribution on Schwartz function spaces.
Similarly, we only consider the action of $\tau$-operator $OP_{\tau}(\s)$ on Schwartz function spaces.

This problem restricted to $s=s_i=t_i=0$, namely,
\ben\label{ieq, without s}
W: \calM^{p_1,q_1}(\rd)\times \calM^{p_2,q_2}(\rd)\longrightarrow M^{p,q}(\rdd)
\een
was studied by Toft \cite[Theorem 4.2]{Toft2004AGAG}, and then refined very recently by Cordero-Nicola \cite[Theorem 1.1]{CorderoNicola2018IMRNI}
and Cordero \cite[Theorem 3.2]{Cordero2020a}.
In \cite{CorderoNicola2018IMRNI,Cordero2020a}, the authors find the sharp conditions for \eqref{ieq, without s} of exponents $p_i, q_i, p, q \in (0,\fy]$, $i=1,2$.
Under the same conditions, they also obtain the following estimate:
\be
W: \calM^{p_1,q_1}_{v_{|s|}}(\rd)\times \calM^{p_2,q_2}_{v_{s}}(\rd)\longrightarrow M^{p,q}_{1\otimes v_{s}}(\rdd).
\ee

In the present paper, our first major goal is to consider BMM problem on a more general level.
For suitable weight function $m, m_1$, $m_2$ on $\rdd$ (see Section 2 for more precise definitions of weights),
our first main theorem shows that BMM can be
characterized by the corresponding convolution inequalities of discrete mixed-norm spaces.

\begin{theorem} (First characterization of BMM)\label{thm-M1}
Assume $p_i, q_i, p, q \in (0,\fy]$, $i=1,2$, $\tau\in [0,1]$.
Suppose that $m, m_i \in \mathscr{P}(\rdd)$, $i=1,2$.
We have
  \be
  W_{\tau}: \calM^{p_1,q_1}_{m_1}(\rd)\times \calM^{p_2,q_2}_{m_2}(\rd)\longrightarrow M^{p,q}_{1\otimes m}(\rdd)
  \ee
if and only if for all $\vec{a},\vec{b}$,
\ben\label{cd0-thm-M1}
\|(a_{k_1,k_2}b_{n_1-k_1,n_2-k_2})\|_{l^{p,q}_{1\otimes m_J}}
\lesssim \|\vec{a}\|_{l^{p_1,q_1}_{m_1}(\zdd)}\|\vec{b}\|_{l^{p_2,q_2}_{\calI m_2}(\zdd)}.
\een
In particular, for $p<\fy$, this is equivalent to
\ben\label{cd1-thm-M1}
l^{p_1/p,q_1/p}_{m_1^p}(\zdd)\ast l^{p_2/p,q_2/p}_{\calI m_2^p}(\zdd)\subset l^{q/p,q/p}_{m_J^p}(\zdd).
\een
Here we write $\calI m_2(z)=m_2(-z)$,
$m_J(z)=m(Jz)=m(z_2,-z_1)$ for $z\in \zdd$, where $J$ is the symplectic matrix (see Section 2).
\end{theorem}

Furthermore, for submultiplicative weight $m$ and
variable-separable weights $m_1$ and $m_2$, namely,
\be
m(z_1+n_1,z_2+n_2)\lesssim m(z_1,z_2)m(n_1,n_2),\ \ \
m_1=\om_1\otimes \mu_1,\ \ \ m_2=\om_2\otimes \mu_2,
\ee
our second main theorem shows that BMM can be
further characterized by some convolution or embedding inequalities of discrete norm spaces.

\begin{theorem}(Second characterization of BMM)\label{thm-M2}
Assume $p_i, q_i, p, q \in (0,\fy]$, $i=1,2$, $\tau\in [0,1]$.
Suppose that $m\in \mathscr{P}(\rdd)$ is submultiplicative, $\om_i, \mu_i \in \mathscr{P}(\rd)$, $i=1,2$.
We have
\be
W_{\tau}: \calM^{p_1,q_1}_{\om_1\otimes \mu_1}(\rd)\times \calM^{p_2,q_2}_{\om_2\otimes \mu_2}(\rd)\longrightarrow M^{p,q}_{1\otimes m}(\rdd)
\ee
if and only if
\begin{eqnarray}
&l_{\om_1^p}^{p_1/p}(\zd)\ast l_{\calI \om_2^p}^{p_2/p}(\zd)\subset  l_{\calI m_{\b}^p}^{q/p}(\zd),\ \ \
l_{\mu_1^p}^{q_1/p}(\zd)\ast l_{\calI \mu_2^p}^{q_2/p}(\zd)\subset  l_{m_{\al}^p}^{q/p}(\zd),\ \ \
&p<\fy,
\\
&l_{\om_1}^{p_1}(\zd),\  l_{\calI \om_2}^{p_2}(\zd)\subset  l_{\calI m_{\b}}^{q}(\zd),\ \ \ \
l_{\mu_1}^{q_1}(\zd),\  l_{\calI \mu_2}^{q_2}(\zd)\subset  l_{m_{\al}}^{q}(\zd),
&p\geq q.
\end{eqnarray}

%
Here, we write $m_{\al}(z_1)=m(z_1,0)$ and $m_{\b}(z_2)=m(0,z_2)$ for $z_1,z_2\in \rd$.
\end{theorem}

As an application, we return to the case of power weights.
Our third characterization shows that in this case BMM
can be characterized by some convolution or embedding inequalities of discrete norm spaces
with power weights. Some further characterizations of exponents will be shown in Section 5.

\begin{theorem}(Third characterization of $BMM$)\label{thm-M3}
  Assume $p, q, p_i, q_i \in (0,\fy]$, $s_i, t_i\in \rr$, $i=1,2$, $\tau\in [0,1]$. We have
\be
W_{\tau}: \calM^{p_1,q_1}_{v_{s_1,t_1}}(\rd)\times \calM^{p_2,q_2}_{v_{s_2,t_2}}(\rd) \longrightarrow M^{p,q}_{1\otimes v_s}(\rdd)
\ee
if and only if
\begin{eqnarray}
&l_{ps_1}^{p_1/p}(\zd)\ast l_{ps_2}^{p_2/p}(\zd)\subset  l_{ps}^{q/p}(\zd),
 \ \ \ l_{pt_1}^{q_1/p}(\zd)\ast l_{pt_2}^{q_2/p}(\zd)\subset  l_{ps}^{q/p}(\zd),
&p<\fy,
\\
&l_{s_1}^{p_1}(\zd),\  l_{s_2}^{p_2}(\zd)\subset  l_{s}^{q}(\zd),\ \ \
 \ \ \ l_{t_1}^{q_1}(\zd),\  l_{t_2}^{q_2}(\zd)\subset  l_{s}^{q}(\zd),
&p\geq q.
\end{eqnarray}

\end{theorem}

It is well known that the boundeness property of $\tau$-pseudodifferential operators with symbols in modulation spaces
are independent with $\tau\in [0,1]$, see \cite[Remark 1.5]{Toft2004JFA}.
As expected, the characterizations of BMM is also independent of $\tau$ as shown in Theorems \ref{thm-M1} to \ref{thm-M3}.
However, situation changes
in the problem of \emph{the boundedness of $\tau$-Wigner distribution acting from weighted modulation spaces to
weighted Wiener amalgam spaces} (BMW).

In this paper, we consider the Wiener amalgam spaces $W(\scrF L^p,L^q_m)(\rdd)$, which
are the image of modulation spaces $M^{p,q}_{1\otimes m}(\rdd)$ under the Fourier transform, see the next section for
its precise definition.
In contrast with the fruitful works on BMM,
there are only few results of BMW. In \cite{DEliaTrapasso2018JPOA,CorderoDEliaTrapasso2019JMAA}, some sufficient conditions of BMW (with weight on the first component of Wiener amalgam space, namely, $W(\scrF L^p_m,L^q)(\rdd)$) was established for $\tau\in (0,1)$, a negative result of a special form of  BMW (see \cite[Proposition 4.4]{DEliaTrapasso2018JPOA}) was
shown for $\tau=0,1$.

Our second major goal is to give some characterizations for BMW on a general level.
For suitable weight function $m, m_1$, $m_2$ on $\rdd$ (see Section 2 for more precise definitions of weights),
our first main theorem for BMW is as follows.

\begin{theorem} (First characterization of BMW)\label{thm-W1}
Assume $p_i, q_i, p, q \in (0,\fy]$, $i=1,2$, $\tau\in [0,1]$. Suppose that $m, m_i \in \mathscr{P}(\rdd)$, $i=1,2$.
Denote $\widetilde{m}(\z_1,\z_2)=m((1-\tau)\z_1, \tau \z_2)$, $\widetilde{m_2}(z_1,z_2)=m_2(\frac{1-\tau}{\tau}z_1, \frac{\tau}{1-\tau}z_2)$.
We have
  \be
  W_{\tau}: \calM^{p_1,q_1}_{m_1}(\rd)\times \calM^{p_2,q_2}_{m_2}(\rd)\longrightarrow W(\scrF L^p, L^q_m)(\rdd)
  \ee
if and only if for all $\vec{a},\vec{b}$,

\begin{eqnarray}
&\label{thm-W1-cd1}
\|(a_{k_1,k_2}b_{n_1-k_1,n_2-k_2})\|_{l^{p,q}_{1\otimes \widetilde{m}}}
\lesssim \|\vec{a}\|_{l^{p_1,q_1}_{m_1}(\zdd)}\|\vec{b}\|_{l^{p_2,q_2}_{\widetilde{m_2}}(\zdd)}
\ \ \ \ &\tau\in (0,1),
\\
&\|(a_{n_1,k_1}b_{k_2,n_2})\|_{l^{p,q}_{1\otimes m}}\lesssim \|\vec{a}\|_{l^{p_1,q_1}_{m_1}(\zdd)}\|\vec{b}\|_{l^{p_2,q_2}_{m_2}(\zdd)}
&\tau=0,
\\
&\|(a_{n_1,k_1}b_{k_2,n_2})\|_{l^{p,q}_{1\otimes m}}\lesssim \|\vec{b}\|_{l^{p_1,q_1}_{m_1}(\zdd)}\|\vec{a}\|_{l^{p_2,q_2}_{m_2}(\zdd)} &\tau=1.
\end{eqnarray}
In particular, for $p<\fy$, the condition \eqref{thm-W1-cd1} is equivalent to
\be
l^{p_1/p,q_1/p}_{m_1^p}(\zdd)\ast l^{p_2/p,q_2/p}_{\widetilde{m_2}^p}(\zdd)
\subset l^{q/p,q/p}_{\widetilde{m}^p}(\zdd).
\ee

\end{theorem}

As in the case of BMM, if $m$ is submultiplicative, $m_1$ and $m_2$ are variable-separable, we obtain a further characterization of BMW.
See the definition of $l^{(p,q)}_m$ in Definition \ref{df-dmII}.

\begin{theorem}(Second characterization of $BMW$)\label{thm-W2}
Assume $p_i, q_i, p, q \in (0,\fy]$, $i=1,2$, $\tau\in [0,1]$.
Suppose that $m\in \mathscr{P}(\rdd)$ is submultiplicative, $\om_i, \mu_i \in \mathscr{P}(\rd)$, $i=1,2$.
Denote $\widetilde{m}(\z_1,\z_2)=m((1-\tau)\z_1, \tau \z_2)$,
$\widetilde{\om_2}(z_1)=\om_2(\frac{1-\tau}{\tau}z_1)$, $\widetilde{\mu_2}(z_2)=\mu_2(\frac{\tau}{1-\tau}z_2)$.
We have
\be
W_{\tau}: \calM^{p_1,q_1}_{\om_1\otimes \mu_1}(\rd)\times \calM^{p_2,q_2}_{\om_2\otimes \mu_2}(\rd)
\longrightarrow W(\scrF L^p, L^q_m)(\rdd)
\ee
if and only if
\begin{eqnarray}
&
l_{\om_1^p}^{p_1/p}(\zd)\ast l_{\widetilde{\om_2}^p}^{p_2/p}(\zd)\subset  l_{\widetilde{m}_{\al}^p}^{q/p}(\zd),\ \
l_{\mu_1^p}^{q_1/p}(\zd)\ast l_{\widetilde{\mu_2}^p}^{q_2/p}(\zd)\subset  l_{\widetilde{m}_{\b}^p}^{q/p}(\zd),\ \ &p<\fy,
\\
&
l_{\om_1}^{p_1}(\zd),\ l_{\widetilde{\om_2}}^{p_2}(\zd)\subset  l_{\widetilde{m}_{\al}}^{q}(\zd),\ \ \
l_{\mu_1}^{q_1}(\zd),\ l_{\widetilde{\mu_2}}^{q_2}(\zd)\subset  l_{\widetilde{m}_{\b}}^{q}(\zd),\ \ \ &p\geq q,
\end{eqnarray}
for $\tau\in (0,1)$, and
\begin{eqnarray}
& l^{p_1,q_1}_{\om_1\otimes\mu_1}(\zdd)\subset l^{(q,p)}_{m_{\al}\otimes 1}(\zdd),
\ l^{p_2}_{\om_2}(\zd)\subset l^p(\zd),\ l^{q_2}_{\mu_2}(\zd)\subset l^q_{m_{\b}}(\zd),\ \
&\tau=0,
\\
& l^{p_2,q_2}_{\om_2\otimes \mu_2}(\zdd)\subset l^{(q,p)}_{m_{\al}\otimes 1}(\zdd),
\ l^{p_1}_{\om_1}(\zd)\subset l^p(\zd),\ l^{q_1}_{\mu_1}(\zd)\subset l^q_{m_{\b}}(\zd),\ \
&\tau=1.
\end{eqnarray}
Here, we write $\widetilde{m}_{\al}(z_1)=\widetilde{m}(z_1,0)$ and $\widetilde{m}_{\b}(z_2)=\widetilde{m}(0,z_2)$ for $z_1,z_2\in \rd$.
\end{theorem}

For the case of power weight, we have following further characterization.
See the characterizations of exponents in Section 5.

\begin{theorem}(Third characterization of $BMW$)\label{thm-W3}
Assume $p_i, q_i, p, q \in (0,\fy]$, $s_i, t_i\in \rr$, $i=1,2$, $\tau\in [0,1]$. We have
\be
W_{\tau}: \calM^{p_1,q_1}_{v_{s_1,t_1}}(\rd)\times \calM^{p_2,q_2}_{v_{s_2,t_2}}(\rd)
\longrightarrow W(\scrF L^p, L^q_s)(\rdd)
\ee
if and only if
\begin{eqnarray}
&
l_{ps_1}^{p_1/p}(\zd)\ast l_{ps_2}^{p_2/p}(\zd)\subset  l_{ps}^{q/p}(\zd),
\ \ \ l_{pt_1}^{q_1/p}(\zd)\ast l_{pt_2}^{q_2/p}(\zd)\subset  l_{pt}^{q/p}(\zd),\ \ &p<\fy,
\\
&
l_{s_1}^{p_1}(\zd),\ l_{s_2}^{p_2}(\zd)\subset  l_{s}^{q}(\zd),
\ \ \ l_{t_1}^{q_1}(\zd),\ l_{t_2}^{q_2}(\zd)\subset  l_{t}^{q}(\zd),\ \ &p\geq q,
\end{eqnarray}
for $\tau\in (0,1)$, and
\begin{eqnarray}
&
l^{q_1}_{t_1}(\zd), l^{p_2}_{s_2}(\zd)\subset l^p(\zd),\ \ \
l^{p_1}_{s_1}(\zd), l^{q_1}_{s_1+t_1}(\zd), l^{q_2}_{t_2}(\zd)\subset l^q_s(\zd),\ \ \
&\tau=0,
\\
&
l^{q_2}_{t_2}(\zd),l^{p_1}_{s_1}(\zd)\subset l^p(\zd),\ \ \
l^{p_2}_{s_2}(\zd), l^{q_2}_{s_2+t_2}(\zd), l^{q_1}_{t_1}(\zd)\subset l^q_s(\zd),\ \ \
&\tau=1.
\end{eqnarray}
\end{theorem}

\bigskip

As mentioned in the beginning of this paper, the boundedness property of Wigner distribution has closed connections
with some important operators, for which we can deduce fruitful new boundedness results from our main Theorems \ref{thm-M1} to \ref{thm-W3}.
Here, we focus on the boundedness of pseudodifferential operators with symbols in modulation and Wiener amalgam spaces.

Let us mention that the study of pseudodifferential operators has a long history in the field of classical harmonic analysis,
we refer the reader to the pioneering works of
Kohn--Nirenberg \cite{KohnNirenberg1965CPAM} and
H\"{o}rmander \cite{Hoermander2007}. See also the famous H\"{o}rmander class in \cite{Hoermander2007}.
The classical Calderon-Vaillancourt theorem \cite{CalderonVaillancourt1971JMSJ}
gives the $L^2$-boundedness of Kohn-Nirenberg operator with symbols belonging to
the H\"{o}rmander's class $S_{0,0}^0$, in which all the derivatives of symbols are required to be bounded.

In the field of time-frequency analysis, the earliest work of pseudodifferential operators
is due to Sj\"{o}strand \cite{Sjoestrand1994MRL}, where the boundednss on $L^2$ of pseudodifferential operators with
symbols in $M^{\fy,1}$ (Sj\"{o}strand's class) was obtained.
Since $S_{0,0}^0\subsetneq M^{\fy,1}$, Sj\"{o}strand's result essentially extended the Calderon-Vaillancourt theorem.
Then, Gr\"{o}chenig--Heil \cite{GroechenigHeil1999IEOT} and Gr\"{o}chenig \cite{Groechenig2006RMI} extended Sj\"{o}strand's result to the boundedness on all modulation spaces $M^{p,q}$ with
$1\leq p,q\leq \fy$.

In this paper, we consider the problems for
\emph{the boundedness on modulation spaces of pseudodifferential operators with symbols in modulation spaces} (BPM)
and
\emph{the boundedness on modulation spaces of pseudodifferential operators with symbols in Wiener amalgam spaces} (BPW).
By an equivalent characterization between BMM (or BMW) and BPM (or BPW), we give several characterizations for BPM and BPW.
See Section 6 for more details.

We also point out that our methods and theorems for BPM and BPW can be extended to the bilinear and even multilinear cases.
See \cite{BenyiOkoudjou2004JFAA, BenyiGroechenigHeilOkoudjou2005JOT} for the boundedness on
modulation spaces of multilinear pseudodifferential operators with symbols in modulation spaces,
and see a recent contribution in \cite{MolahajlooOkoudjouPfander2016JFAA} for symbols in some modified modulation spaces.
We may revisit this topic of multilinear cases in the future.

The rest of this paper is organized as follows. In Section 2, we recall some definitions
of function spaces we shall use. We also list some basic time-frequency representations
associated with Wigner distribution, and recall the Gabor expansion of modulation spaces, which
are the key tools for our first characterizations in Theorems \ref{thm-M1} and \ref{thm-W1}.

Section 3 is devoted to the first characterizations of BMM and BMW.
First, Theorem \ref{thm-M1} is proved by the help of the time-frequency tools mentioned in Section 2.
Then, we establish the relations between BMM and BMW in Proposition 3.3.
Combining this with Theorem \ref{thm-M1}, we give the proof for the non-endpoint case of Theorem \ref{thm-W1}.
Like Theorem \ref{thm-M1}, the endpoint case of \ref{thm-W1} will be proved directly by the time-frequency tools.

In Section 4, under some reasonable assumptions of weights, we give further characterizations of BMM and BMW.
The separation of convolution inequality, i.e. Proposition \ref{pp-sepc},
 yields the proof for Theorem \ref{thm-M2} and the non-endpoint case of Theorem \ref{thm-W2}.
The separation of mixed-norm embedding inequality, i.e. Proposition \ref{pp-sepe},
 yields the proof for the endpoint case of Theorem \ref{thm-W2}.

The power weight case will be handled in Section 5.
The proof for Theorem \ref{thm-M3} and the non-endpoint case of Theorem \ref{thm-W3} follows directly by Theorems \ref{thm-M2} and
\ref{thm-W2}. The proof of endpoint case of Theorem \ref{thm-W3} follows
by the further separation of mixed-norm embedding, namely, Proposition \ref{pp-sep2-eb}.
We also list Lemmas \ref{lm-exp-wcov}, \ref{lm-exp-cov} and \ref{lm-exp-eb} for
further exponent characterizations. Then, several characterizations of exponents
are established for BMM and BMW, see Theorems \ref{thm-M4}, \ref{thm-W4}, \ref{thm-M5}, \ref{thm-W5}, \ref{thm-M6} and \ref{thm-W6}
for the sharp exponents of BMM and non-endpoint cases of BMW.
See Theorem \ref{thm-W7} for the sharp exponents of endpoint cases of BMW.

In Section 6, by establishing some equivalent relations in Propositions \ref{pp-eqOP-BMM} and \ref{pp-eqOP-BMW}, we give several
useful characterizations of BPM and BPW. In particular, the sharp exponents of unweighted version of BPM and BPW will be
given in Theorems \ref{thm-OPM} and \ref{thm-OPWE}.
The sharp exponents for BPM with Sj\"{o}strand's class and for BPW with symbols in $W(\scrF L^1, L^{\fy})(\rdd)$ can be
founded in Theorems \ref{thm-OPM-SJ}, \ref{thm-OPW-SJ} and \ref{thm-OPWE-SJ}, and Remarks \ref{rk_cpSJ} and \ref{rk_cpSJE}
will be prepared for the comparisons between them.
At the end of this section, we give the sharp exponents for the boundedness on Sobolev spaces $H^s$ of pseudodifferential operators with
symbols in Wiener amalgam spaces, showing that the boundedness on Sobolev spaces can not happen with $W(\scrF L^{\fy}, L^1)(\rdd)$ symbols.

\textbf{Notations:}
Throughout this paper, we will adopt the following notations. Let $C$ be a
positive constant that may depend on $d, p, q, p_{i},\,q_{i},\,s_{i},\,t_{i}, m, m_i,
\omega _{i}, \mu_{i},\,(i=1,\,2)$. The notation $X\lesssim Y$ denotes the
statement that $X\leq CY$, and The notation $X\sim Y$ means the statement $%
X\lesssim Y\lesssim X$.
The Schwartz function space is denoted by $\calS(\rd)$, and the space of tempered distributions by $\calS'(\rd)$.
We use the brackets $\langle f,g\rangle$ to denote
the extension to $\calS'(\rd)\times \calS(\rd)$ of the inner product $\langle f,g\rangle=\int_{\rd} f(x)\overline{g(x)}dx$ for $f,g\in L^2(\rd)$.
We set $\calI f(x)=f(-x)$ and $\calD_{\lambda}f(x)=f(\lambda x)$ for $\la\in \rr$, $x\in \rd$.

\section{PRELIMINARIES}

\subsection{Time-frequency representations}
The translation operator $T_x$ and modulation operator $M_{\xi}$ are defined as
\be
T_xf(t)=f(t-x),\ \ \ \ M_{\xi}f(t)=e^{2\pi it\xi}f(t).
\ee
We recall that, as a bilinear map on $L^2(\rd)\times L^2(\rd)$,
the STFT $V_gf$ can be extended to be a map from $\calS'(\rd)\times \calS(\rd)$ into $\calS'(\rdd)$ by
\be
V_gf(x,\xi)=\langle f, M_{\xi}T_xg\rangle.
\ee
In fact, for $f\in \calS'(\rd)$ and $g\in \calS(\rd)$, $V_gf$ is a continuous function on $\rdd$ with polynomial growth,
see \cite[Theorem 11.2.3]{GrochenigBook2013}.
The so-called fundamental indentity of time-frequency analysis is as follows:
\be
V_gf(x,\xi)=e^{-2\pi ix\cdot \xi}V_{\hat{g}}\hat{f}(\xi,-x),\ \ \ (x,\xi)\in \rdd.
\ee
Next, we calculate the linear transform of STFT.

\begin{lemma}[Linear transform of STFT]\label{lm-bpSTFT}
Assume $f,g\in L^2(\rd)$. Let $L$ be a invertible linear transform on $\rd$.
For a function $f$, denote $f_L(x):=f(Lx)$.
We have
\be
V_{\phi_L}f_L(x,\xi)=|\det(L)|^{-1}V_{\phi}f(Lx,(L^{-1})^T\xi).
\ee
\end{lemma}

\begin{proof}
By a direct calculation, we have
  \be
  \begin{split}
  V_{\phi_L}f_L(x,\xi)
  = &
  \int_{\rd}f(Lt)\overline{\phi(Lt-Lx)}e^{-2\pi iLt\cdot (L^{-1})^T\xi}dt
  \\
  = &
  |\det(L)|^{-1}\int_{\rd}f(t)\overline{\phi(t-Lx)}e^{-2\pi it\cdot (L^{-1})^T\xi}dy
  =
  |\det(L)|^{-1}V_{\phi}f(Lx,(L^{-1})^T\xi).
  \end{split}
  \ee
\end{proof}

In the next lemma, we calculate the STFTs of of $\tau$-Wigner distributions, which are the key tools for the estimates
of $\tau$-Wigner distributions on modulation spaces.
We refer the readers to \cite{CorderoDEliaTrapasso2019JMAA} for the process of calculations.
\begin{lemma}[STFT of $\tau$-Wigner distribution]\label{lm-STFT-tWd}
Consider $\tau\in [0,1]$.
Let $\Phi_{\tau}=W_{\tau}(\phi_1,\phi_2)$ for nonzero functions $\phi_1, \phi_2\in \calS(\rd)$. Then the STFT of $W_{\tau}(f_1,f_2)$ with respect to
the window $\Phi_{\tau}$ is given by
\be
V_{\P_{\tau}}(W_{\tau}(f_1,f_2))(z,\z)
=e^{-2\pi iz_2\z_2}V_{\phi_1}f_1(z_1-\tau \z_2,z_2+(1-\tau)\z_1)\overline{V_{\phi_2}f_2(z_1+(1-\tau)\z_2,z_2-\tau \z_1)}.
\ee
In particular, for $\tau=0$,
\be
V_{\P_{0}}(W_{0}(f_1,f_2))(z,\z)
=e^{-2\pi iz_2\z_2}V_{\phi_1}f_1(z_1,z_2+\z_1)\overline{V_{\phi_2}f_2(z_1+\z_2,z_2)}.
\ee
For $\tau=1$, we have
\be
V_{\P_{1}}(W_{1}(f_1,f_2))(z,\z)
=e^{-2\pi iz_2\z_2}V_{\phi_1}f_1(z_1-\z_2,z_2)\overline{V_{\phi_2}f_2(z_1,z_2-\z_1)}.
\ee
For $\tau=\frac{1}{2}$, we have
\be
\begin{split}
V_{\P}(W(f_1,f_2))(z,\z)
= &
e^{-2\pi iz_2\z_2}V_{\phi_1}f_1(z_1-\frac{\z_2}{2},z_2+\frac{\z_1}{2})\overline{V_{\phi_2}f_2(z_1+\frac{\z_2}{2},z_2-\frac{\z_1}{2})},
\\
= &
e^{-2\pi iz_2\z_2}V_{\phi_1}f_1(z-\frac{1}{2}J\z)\overline{V_{\phi_2}f_2(z+\frac{1}{2}J\z)},
\end{split}
\ee
where $J$ is the canonical symplectic matrix in $\rr^{2d}$ defined by
\be
J=
\begin{pmatrix}

    0_{d\times d} & I_{d\times d} \\

    -I_{d\times d} & 0_{d\times d}

\end{pmatrix}.\qquad
\ee
\end{lemma}

\begin{lemma}[Connection between STFT and $\tau$-Wigner distribution I]\label{lm, STFT-twd1}
For $\tau\in (0,1)$, $f_1,f_2\in L^2(\rd)$, we have
\be
W_{\tau}(f_1,f_2)(x,\xi)
=\tau^{-d}e^{2\pi i\tau^{-1}x\cdot \xi}V_{\calD_{\frac{1-\tau}{\tau}}\calI f_2}f_1(\frac{1}{1-\tau}x,\frac{1}{\tau}\xi).
\ee
\end{lemma}
\begin{proof}
  \be
  \begin{split}
    W_{\tau}(f_1,f_2)(x,\xi)
    = &
    \int_{\rd}f_1(x+\tau t)\overline{f_2(x-(1-\tau)t)}e^{-2\pi it\cdot \xi}dt
    \\
    = &
    \int_{\rd} f_1(x+\tau t)\overline{(\calD_{\frac{1-\tau}{\tau}}\calI f_2)(\tau t-\frac{\tau}{1-\tau}x)}e^{-2\pi it\cdot \xi}dt
    \\
    = &
    \tau^{-d}\int_{\rd} f_1(x+t)\overline{(\calD_{\frac{1-\tau}{\tau}}\calI f_2)(t-\frac{\tau}{1-\tau}x)}e^{-2\pi i\tau^{-1}t\cdot \xi}dt
    \\
    = &
    \tau^{-d}e^{2\pi i\tau^{-1}x\cdot \xi}\int_{\rd} f_1(t)\overline{(\calD_{\frac{1-\tau}{\tau}}\calI f_2)(t-\frac{1}{1-\tau}x)}e^{-2\pi i\tau^{-1}t\cdot \xi}dt
    \\
    = &
    \tau^{-d}e^{2\pi i\tau^{-1}x\cdot \xi}V_{\calD_{\frac{1-\tau}{\tau}}\calI f_2}f_1(\frac{1}{1-\tau}x,\frac{1}{\tau}\xi).
  \end{split}
  \ee
\end{proof}

\begin{lemma}[Connection between STFT and $\tau$-Wigner distribution II]\label{lm, STFT-twd2}
  \be
  \scrF W_{\tau}(f_1,f_2)(z)=e^{-2\pi i\tau z_1\cdot z_2}V_{f_2}f_1(-Jz)
  \ee
\end{lemma}
\begin{proof}
  Write
  \be
  W_{\tau}(f_1,f_2)(x,\xi)=\scrF_2^{-1}(f(x-\tau \cdot)\overline{f_2(x+(1-\tau)\cdot)})(\xi).
  \ee
  Then,
  \be
  \begin{split}
    \scrF W_{\tau}(f_1,f_2)(z)
    = &
    \scrF_1(f_1(\cdot-\tau z_2)\overline{f_2(\cdot+(1-\tau)z_2)})(z_1)
    \\
    = &
    \int_{\rd}f_1(x-\tau z_2)\overline{f_2(x+(1-\tau)z_2)}e^{-2\pi ix\cdot z_1}dx
    \\
    = &
    e^{-2\pi i\tau z_1\cdot z_2}\int_{\rd}f_1(x)\overline{f_2(x+z_2)}e^{-2\pi ixz_1}dx
    \\
    = &
    e^{-2\pi i\tau z_1\cdot z_2}V_{f_2}f_1(-z_2,z_1)=e^{-2\pi i\tau z_1\cdot z_2}V_{f_2}f_1(-Jz).
  \end{split}
  \ee
\end{proof}

\subsection{Function spaces}
As we mentioned above, modulation spaces are defined as a measure of the STFT of $f\in \calS'$.
In order to draw a more accurate portrait of the decay and summability properties of STFT, modulation space
is usually be measured by the weighted norm. Let us recall the definitions of weights we shall use.

A weight function $v$ on $\rd$ is called submultiplicative
if $v(z_1+z_2)\leq v(z_1)v(z_2)$ for all $z_1,z_2\in \rd$, a weight function $m$ on $\rd$ is called $v$-moderate
if
\be
m(z_1+z_2)\leq Cv(z_1)m(z_2),\ \ \ \ z_1,z_2\in \rd.
\ee
In this paper, we consider the $v$-moderate weights where $v$ is submultiplicative with polynomial growth.
We use the notation $\scrP(\rd)$ to denote the cone of all non-negative functions which are
$v$-moderate. Similarly, we can define $\scrP(\rdd)$.

The weighted mixed-norm spaces used to measure the STFT is defined as following.
\begin{definition}[Weighted mixed-norm spaces.]
Let $m\in \scrP(\rdd)$, $p,q\in (0,\fy]$. Then the weighted mixed-norm space $L^{p,q}_m(\rdd)$
consists of all Lebesgue measurable functions on $\rdd$ such that the (quasi-)norm
\be
\|F\|_{L^{p,q}_m(\rdd)}=\left(\int_{\rd}\left(\int_{\rd}|F(x,\xi)|^pm(x,\xi)^pdx\right)^{q/p}d\xi\right)^{1/q}
\ee
is finite, with usual modification when $p=\fy$ or $q=\fy$.
\end{definition}
Now, we recall the definition of modulation space.
\begin{definition}\label{df-M}
Let $0<p,q\leq \infty$, $m\in \mathscr{P}(\rdd)$.
Given a non-zero window function $\phi\in \calS(\rd)$, the (weighted) modulation space $M^{p,q}_m(\rd)$ consists
of all $f\in \calS'(\rd)$ such that the norm
\be
\begin{split}
\|f\|_{M^{p,q}_m(\rd)}&:=\|V_{\phi}f(x,\xi)\|_{L^{p,q}_m(\rdd)}
=\left(\int_{\rd}\left(\int_{\rd}|V_{\phi}f(x,\xi)m(x,\xi)|^{p} dx\right)^{{q}/{p}}d\xi\right)^{{1}/{q}}
\end{split}
\ee
is finite.
We write $M^{p,q}$ for modulation space with $m\equiv 1$.
\end{definition}
Recall that the above definition of $M^{p,q}_m$ is independent of the choice of window function $\phi$.
The readers may see this fact in \cite{GrochenigBook2013} for the case $(p,q)\in\lbrack 1,\infty ]^{2}$,
and in \cite{GalperinSamarah2004ACHA} for full range $(p,q)\in (0,\infty ]^{2}$.
In particular, we point out that in \cite{GalperinSamarah2004ACHA}, the author find a admissible windows, denoted by $\mathfrak{M}^{p,q}_v$,
depending on $p, q$,
for the modulation space $M^{p,q}_m$.

Denote by $\calM^{p,q}_m(\rd)$ the $\calS(\rd)$ closure in $M^{p,q}_m(\rd)$.
Note that $\calM^{p,q}_m(\rd)=M^{p,q}_m(\rd)$ for $p,q\neq \fy$.

Next, we turn to the definition of Wiener amalgam space.
As we mentioned in Section 1, in this paper we consider the Wiener amalgam space of the type of the image of modulation space under Fourier transform.
Write
\be
\begin{split}
\|f\|_{\scrF M^{p,q}_m(\rd)}
= &
\|\scrF^{-1}f\|_{M^{p,q}_m(\rd)}
\\
= &
\|V_{\check{\phi}}\check{f}(x,\xi)\|_{L^{p,q}_m(\rdd)}
=
\|V_{\phi}f(\xi,-x)\|_{L^{p,q}_m(\rdd)}=\|(V_{\phi}f)(J(x,\xi))\|_{L^{p,q}_m(\rdd)}.
\end{split}
\ee
The Wiener amalgam space can be also defined by the weighted mixed-norm of STFT.
\begin{definition}\label{df-W}
Let $0<p,q\leq \infty$, $m\in \mathscr{P}(\rdd)$.
Given a non-zero window function $\phi\in \calS(\rd)$, the (weighted) Wiener amalgam space $\scrF(M^{p,q}_m)$ consists
of all $f\in \calS'(\rd)$ such that the norm
\be
\begin{split}
\|f\|_{\scrF M^{p,q}_m(\rd)}
=
\|V_{\phi}f(\xi,-x)\|_{L^{p,q}_m(\rdd)}
= \left(\int_{\mathbb{R}^n}\left(\int_{\mathbb{R}^n}|V_{\phi}f(\xi,-x)m(x,\xi)|^{p} dx\right)^{{q}/{p}}d\xi\right)^{{1}/{q}}
\end{split}
\ee
is finite.
\end{definition}
In particular, for $f\in \calS'(\rd)$, $\phi\in \calS(\rd)$,
\be
\|f\|_{_{\scrF M^{p,q}_{1\otimes m}(\rdd)}}
=
\left(\int_{\rd}\left(\int_{\rd}|V_{\phi}f(\xi,-x)|^{p} dx\right)^{{q}/{p}}m(\xi)^qd\xi\right)^{{1}/{q}}.
\ee
Using the notation of Wiener amalgam space in \cite{FeichtingerIPCoFSOB1}, we have
\be
\scrF M^{p,q}_{1\otimes m}(\rdd)=W(\scrF L^p,L^q_m)(\rdd).
\ee
This representation makes us more clear that $f$ belongs to $\scrF M^{p,q}_{1\otimes m}(\rdd)$ means it lies locally
in $\scrF L^p(\rdd)$ and globally in $L^q_m(\rdd)$.

Next, we collect following calculations for the linear transform of modulation and Wiener amalgam spaces.
\begin{lemma}\label{lm-ltM}
Let $0<p,q\leq \fy$, $L$ be a invertible linear transform on $\rd$, $m\in\scrP(\rdd)$. For a function $f$ on $\rd$, denote $f_L(x):=f(Lx)$.
We have
\be
\|f_L\|_{M^{p,q}_m(\rd)}\sim \|f\|_{M^{p,q}_{\widetilde{m}}(\rd)},
\ee
where $\widetilde{m}(x,\xi)=m(L^{-1}x,L^T\xi)$, $x, \xi\in \rd$, $i=1,2$.
\end{lemma}
\begin{proof}
  Using Lemma \ref{lm-bpSTFT}, we write
  \be
  \begin{split}
    \|f_L\|_{M^{p,q}_m}
    = &\|V_{\phi_L}f_L(x,\xi)\|_{L^{p,q}_m}
    \\
    \sim &
    \|V_{\phi}f(Lx,(L^{-1})^T\xi)\|_{L^{p,q}_m}
    \sim
    \|V_{\phi}f(x,\xi)\|_{L^{p,q}_{\widetilde{m}}}\sim \|f\|_{M^{p,q}_{\widetilde{m}}}.
  \end{split}
  \ee
\end{proof}

\begin{lemma}\label{lm-ltW}
Let $0<p,q\leq \fy$, $L$ be a invertible linear transform on $\rdd$, $m\in\scrP(\rdd)$. For a function $f$ on $\rdd$, denote $f_L(x):=f(Lx)$.
We have
\be
\|f_L\|_{W(\scrF L^p,L^q_m)(\rdd)}\sim \|f\|_{W(\scrF L^p,L^q_{\widetilde{m}})(\rdd)},
\ee
where $\widetilde{m}(z)=m(L^{-1}z)$.
\end{lemma}
\begin{proof}
  By Lemma \ref{lm-bpSTFT} and Definition \ref{df-W}, we have
  \be
  \begin{split}
    \|f_L\|_{W(\scrF L^p,L^q_m)(\rdd)}
    = &
    \|V_{\phi_L}f_L(\z,-z)\|_{L^{p,q}_{1\otimes m}}
    \\
    \sim &
    \|V_{\phi}f(L\z,-(L^{-1})^Tz)\|_{L^{p,q}_{1\otimes m}}
    \\
    \sim &
    \|V_{\phi}f(\z,-z)\|_{L^{p,q}_{1\otimes \widetilde{m}}}\sim \|f\|_{W(\scrF L^p,L^q_{\widetilde{m}})(\rdd)}.
  \end{split}
  \ee
\end{proof}

Next, we recall a multiplication property of Wiener amalgam space.
\begin{lemma}\label{lm-mpW}
  Let $0<p,q\leq \fy$, $\dot{p}=\min\{p,1\}$.
  We have
  \be
  W(\scrF L^p,L^q_m)\cdot W(\scrF L^{\dot{p}},L^{\fy})\subset W(\scrF L^p,L^q_m).
  \ee
\end{lemma}
\begin{proof}
Using the relation between modulation and Wiener amalgam space, the desired conclusion
is equivalent to
\be
M^{p,q}_{1\otimes m}(\rdd)\ast M^{\dot{p},\fy}\subset M^{p,q}_{1\otimes m}(\rdd).
\ee
This is a direct conclusion of \cite[Theorem 1.3]{GuoChenFanZhao2019MMJ} with the fact
\be
l^p\ast l^{\dot{p}}\subset l^p,\ \ \ l^q_m\cdot l^{\fy}\subset l^q_m.
\ee
\end{proof}

\begin{lemma}[Chirp function]\label{lm-chirp}
  For any $\la\neq 0$, we have
  \be
  G_{\la}(x,\xi)=e^{2\pi i\la x\cdot\xi}\in W(\scrF L^{\dot{p}},L^{\fy})(\rdd).
  \ee
  Moreover, for any function $F$ on $\rdd$, we have
  \be
  \|FG_\la\|_{W(\scrF L^p, L^q_m)(\rdd)}\sim \|F\|_{W(\scrF L^p, L^q_m)(\rdd)}.
  \ee
\end{lemma}
\begin{proof}
  Using Lemma \ref{lm-ltW}, we only need to consider the case $\la=1$, i.e., to verify
  that $G_1=e^{2\pi ix\cdot \xi}\in W(\scrF L^{\dot{p}}, L^q_m)$.
  Denote by $g_0(x,\xi)=e^{-\pi(|x|^2+|\xi|^2)}$ the Gaussian function.
  By the calculation in \cite[Proposition 3.2]{CorderoGossonNicola2018ACHA}, we have
  \be
  |V_{g_0}G_1(z,\z)|=2^{-d/2}e^{-\frac{\pi}{2}|z_1-\z_2|^2}e^{-\frac{\pi}{2}|z_2-\z_1|^2}.
  \ee
  By Definition \ref{df-W}, we conclude that
  \be
  \begin{split}
  \|G_1\|_{W(\scrF L^{\dot{p}}, L^{\fy})(\rdd)}
  = &
  \|V_{g_0}G_1(\z,-z)\|_{L^{\dot{p},\fy}(\rddd)}
  \\
  \sim &
  \|e^{-\frac{\pi}{2}|\z_1+z_2|^2}e^{-\frac{\pi}{2}|\z_2+z_1|^2}\|_{L^{\dot{p},\fy}(\rddd)}\sim 1.
  \end{split}
  \ee
  This completes the proof of $G_{\la}\in W(\scrF L^{\dot{p}},L^{\fy})$.

  Moreover, by the multiplication property (Lemma \ref{lm-mpW}), we deduce that
  \be
  \begin{split}
  \|FG_{\la}\|_{W(\scrF L^p, L^q_m)}
  \lesssim &
  \|F\|_{W(\scrF L^p, L^q_m)} \|G_{\la}\|_{W(\scrF L^{\dot{p}},L^{\fy})}
  \\
  \lesssim &
  \|F\|_{W(\scrF L^p, L^q_m)}
  =
  \|FG_{\la}G_{-\la}\|_{W(\scrF L^p, L^q_m)}\lesssim \|FG_{\la}\|_{W(\scrF L^p, L^q_m)}.
  \end{split}
  \ee
  From this, we obtain $\|FG_{\la}\|_{W(\scrF L^p, L^q_m)}\sim \|F\|_{W(\scrF L^p, L^q_m)}$.
\end{proof}

In order to measure the summability and decay properties of Gabor coefficients, we recall the discrete weighted mixed-norm space.

\begin{definition}[Discrete mixed-norm spaces I]\label{df-dmI}
Let $0<p,q\leq \fy$, $m\in \scrP(\rdd)$.
The space $l^{p,q}_m(\zdd)$ consists of all sequences $\vec{a}=\{a_{k,n}\}_{k,n\in \zd}$ for which the (quasi-)norm
\be
\|\vec{a}\|_{l^{p,q}_m(\zdd)}=\left(\sum_{n\in \zd}\left(\sum_{k\in \zd}|a_{k,n}|^pm(k,n)^p\right)^{q/p}\right)^{1/q}
\ee
is finite.
\end{definition}

In Theorem \ref{thm-W2}, we use another type of discrete weighted mixed-norm space as following.
\begin{definition}[Discrete mixed-norm spaces II]\label{df-dmII}
Let $0<p,q\leq \fy$, $m\in \scrP(\rdd)$.
The space $l^{(p,q)}_m(\zdd)$ consists of all sequences $\vec{a}=\{a_{k,n}\}_{k,n\in \zd}$ for which the (quasi-)norm
\be
\|\vec{a}\|_{l^{(p,q)}_m(\zdd)}=\left(\sum_{k\in \zd}\left(\sum_{n\in \zd}|a_{k,n}|^qm(k,n)^q\right)^{p/q}\right)^{1/p}
\ee
is finite.
\end{definition}
As usual for $\om\in \scrP(\rd)$, the space $l^p_{\om}(\zd)$ consists of all $\vec{b}=\{b_{k}\}_{k\in \zd}$ for which the (quasi-)norm
\be
\|\vec{b}\|_{l^{p}_{\om}(\zd)}=\left(\sum_{k\in \zd}|b_{k}|^p\om(k)^p\right)^{1/p}
\ee
is finite. For $\om=v_s$, we write $l^p_{v_s}:= l^p_s$ for simplicity.

\subsection{Gabor analysis of modulation spaces}
Comparing with the classical definition of modulation space in Definition \ref{df-M},
or the semi-discrete definition such as in \cite[Proposition 2.1]{GuoChenFanZhao2019MMJ} in the same way as Besov spaces,
the modulation spaces can be also characterized by the summability and decay properties of their Gabor coefficients,
this is an important reason why the modulation spaces play the central role in the field of time-frequency analysis.

We recall some important operators which are the key tools for the discretization of modulation spaces.

\begin{definition}
  Assume that $g,\g\in L^2(\rd)$ and $\al,\b>0$. The coefficient operator or analysis operator $C_g^{\al,\b}$
  is defined by
  \be
  C_g^{\al,\b}f=(\langle f, T_{\al k}M_{\b n}g\rangle)_{k,n\in \zd}.
  \ee
  The synthesis operator or reconstruction operator $D_{g}^{\al,\b}$ is defined by
  \be
  D_{\g}^{\al,\b}\vec{c}=\sum_{k\in \zd}\sum_{n\in \zd}c_{k,n}T_{\al k}M_{\b n}\g.
  \ee
  The Gabor frame operator $S_{g,\g}^{\al,\b}$ is defined by
  \be
  S_{g,\g}^{\al,\b}f=D_{\g}^{\al,\b}C_g^{\al,\b}f=\sum_{k\in \zd}\sum_{n\in \zd}\langle f, T_{\al k}M_{\b n}g\rangle T_{\al k}M_{\b n}\g.
  \ee
\end{definition}

In order to extend the boundedness result of analysis operator and synthesis operator to the modulation spaces of full range,
following admissible window class was introduced in \cite{GalperinSamarah2004ACHA}.

\begin{definition}[The space of admissible windows]\label{df-space-windows}
Assume that $m$ is $v$-moderate and let $0<p,q\leq \fy$. Let $r=\min\{1,p\}$ and $s=\min\{1,p,q\}$.
For $r_1,s_1>0$, denote
\be
\om_{r_1,s_1}(x,\om)=v(x,\om)\cdot (1+|x|)^{r_1}\cdot(1+|\om|)^{s_1}.
\ee
Define the space of admissible windows $\mathfrak{M}^{p,q}_v$ for the modulation space $M^{p,q}_m$ to be
\be
\mathfrak{M}^{p,q}_v=\bigcup_{\substack {r_1>d/r \\ s_1>d/s\\1\leq p_1<\fy}}M^{p_1}_{\om_{r_1,s_1}}.
\ee
\end{definition}

Based on the window class mentioned above, we recall the boundedness of $C_g^{\al,\b}$ and $D_{g}^{\al,\b}$,
which works on the full range $p,q\in (0,\fy]$.

\begin{lemma}\label{lm, bdCD}
  Assume that $m$ is $v$-moderate, $p,q\in (0,\fy]$, and $g$ belongs to the subclass $M^{p_1}_{\om_{r_1,s_1}}$ of $\mathfrak{M}^{p,q}_v$.
  Then the analysis operator $C_g^{\al,\b}$ is boundedness from $M^{p,q}_m$ into $l^{p,q}_{\tilde{m}}$,
  and the synthesis operator $D_g^{\al,\b}$ is boundedness form $l^{p,q}_{\tilde{m}}$ into $M^{p,q}_m$ for all $\al,\b>0$,
  where $\tilde{m}(k,n)=m(\al k,\b n)$.
\end{lemma}

Now, we recall the main theorem in \cite{GalperinSamarah2004ACHA}, which extends the Gabor expansion of modulation spaces to the full range $0<p,q\leq \fy$.

\begin{theorem}\label{thm, frame for Mpq}(see \cite{GalperinSamarah2004ACHA})
  Assume that $m$ is $v$-moderate, $p,q\in (0,\fy]$, $g,\g\in \mathfrak{M}^{p,q}_v$, and that the Gabor frame operator
  $S^{\al,\b}_{g,\g}=D^{\al,\b}_{\g}C^{\al,\b}_g=I$ on $L^2(\rd)$. Then
  \be
  f=\sum_{k\in \zd}\sum_{n\in \zd}\langle f, T_{\al k}M_{\b n}g\rangle T_{\al k}M_{\b n}\g
  =\sum_{k\in \zd}\sum_{n\in \zd}\langle f, T_{\al k}M_{\b n}\g\rangle T_{\al k}M_{\b n}g
  \ee
  with unconditional convergence in $M^{p,q}_m$ if $p,q<\fy$, and with weak-star convergence in $M^{\fy}_{1/v}$ otherwise.
  Furthermore there are constants $A,B>0$ such that for all $f\in M^{p,q}_m$
  \be
  A\|f\|_{M^{p,q}_m}
  \leq
  \left(\sum_{n\in \zd}\left(\sum_{k\in \zd}|\langle f, T_{\al k}M_{\b n}g\rangle|^pm(\al k,\b n)^p \right)^{q/p}\right)^{1/q}
  \leq
  B\|f\|_{M^{p,q}_m}
  \ee
  with obvious modification for $p=\fy$ or $q=\fy$.
  Likewise, the quasi-norm equivalence
    \be
  A'\|f\|_{M^{p,q}_m}
  \leq
  \left(\sum_{n\in \zd}\left(\sum_{k\in \zd}|\langle f, T_{\al k}M_{\b n}\g\rangle|^pm(k\al,n\b)^p \right)^{q/p}\right)^{1/q}
  \leq
  B'\|f\|_{M^{p,q}_m}
  \ee
  holds on $M^{p,q}_m$.
  \end{theorem}

  The following well known theorem provide a way to find the Gabor frame of $L^2(\rd)$.
  Recall that $\|g\|_{W(L^{\fy},L^{1})(\rd)}=\sum_{n\in \zd}\|g\chi_{Q_0+n}\|_{L^{\fy}}$ with $Q_0=[0,1]^d$.
\begin{theorem}(Walnut \cite{Walnut1992JMAA})\label{thm-frame-L^2}
  Suppose that $g\in W(L^{\fy},L^{1})(\rd)$ satisfying
  \be
  A\leq \sum_{k\in \zd}|g(x-\al k)|^2\leq B\ \ \ \ a.e.
  \ee
  for constants $A,B\in (0,\fy)$.
  Then there exists a constant $\b_0$ depending on $\al$ such that
  $\mathcal {G}(g,\al,\b)$ is a Gabor frame of $L^2(\rd)$ for all $\b\leq \b_0$.
\end{theorem}

In order to find the dual window in a suitable function space, the following result is important.
\begin{theorem}(\cite{GrochenigBook2013})\label{thm-frame-invertible}
  Assume that $g\in M^1_v(\rd)$ and that $\{T_{\al k}M_{\b n}g\}_{k,n\in \zd}$ is a Gabor frame for $L^2(\rd)$.
  Then the Gabor frame operator $S_{g,g}^{\al,\b}$ is invertible on $M^1_v(\rd)$. As a consequence, $S_{g,g}^{\al,\b}$
  is invertible on all modulation spaces $M^{p,q}_m(\rd)$ for $1\leq p,q\leq \fy$ and $m\in \scrP(\rdd)$.
\end{theorem}

\section{First characterizations: discretizations by time-frequency tools}
\subsection{Discretization by Gabor coefficients for BMM}
\begin{proposition}\label{pp-M1}
Assume $p_i, q_i, p, q \in (0,\fy]$, $i=1,2$. For any $\al>0$, we have
  \be
  W_0: \calM^{p_1,q_1}_{m_1}(\rd)\times \calM^{p_2,q_2}_{m_2}(\rd) \longrightarrow M^{p,q}_{1\otimes m}(\rdd)
  \ee
if and only if
\ben\label{pp-M1-4}
  \|(a(k)b(k+Jn))\|_{l^{p,q}_{1\otimes \widetilde{m}}(\zdd\times \zdd)}
\lesssim \|(a(k))\|_{l^{p_1,q_1}_{\widetilde{m_1}}(\zdd)}\cdot \|(b(k))\|_{l^{p_2,q_2}_{\widetilde{m_2}}(\zdd)}.
\een
Here, we denote $\widetilde{m}(z): =m(\al z)$ and $\widetilde{m_i}(z): =m_i(\al z)$ for $z\in \zdd$, $i=1,2$.
\end{proposition}
\begin{proof}
We divide the proof into two parts.

\textbf{``Only if'' part.}
  Let $\varphi$ be a smooth function supported in $(-\al/2,\al/2)^d$,
   satisfying $\hat{\varphi}(x)\geq 0$ and $\hat{\varphi}(0)=1$.
  For any $\al>0$ and two nonnegative truncated (only finite nonzero items) sequences
  $\vec{a}=(a_{j,l})_{j,l\in \zd}$ and $\vec{b}=(b_{j,l})_{j,l\in \zd}$, we set
  \be
  f_1=D_{\va}^{\al,\al}\vec{a}=\sum_{j,l\in \zd}a_{j,l}T_{\al j} M_{\al l}\va,
  \ \ \ f_2=D_{\va}^{\al,\al}\vec{b}=\sum_{j,l\in \zd}b_{j,l}T_{\al j} M_{\b l}\va.
  \ee
  Take the window $\Phi=W_0(\phi,\phi)$ with $\phi\in \calS(\rd)$ supported in $(-\al/2,\al/2)^{d}$,
   satisfying $\hat{\phi}(x)\geq 0$ and $\hat{\phi}(0)=1$.
  By the sampling property of STFT (see Lemma \ref{lm, bdCD}),
  \ben\label{pp-M1-1}
  \|V_{\P}(W_0(f_1,f_2))|_{\al\zdd\times \al\zdd}\|_{l^{p,q}_{1\otimes \widetilde{m}}}
  \lesssim
  \|V_{\P}(W_0(f_1,f_2))\|_{L^{p,q}_{1\otimes m}}=\|W_0(f_1,f_2)\|_{M^{p,q}_{1\otimes m}}.
  \een
  Let us turn to the lower bound estimates of the first term in \eqref{pp-M1-1}.
  Using Lemma \ref{lm-STFT-tWd}, we write
  \be
  \begin{split}
    \|V_{\P}(W_0(f_1,f_2))|_{\al\zdd\times \al\zdd}\|_{l^{p,q}_{1\otimes \widetilde{m}}}
    = &
    \|V_{\phi}f_1(z_1,z_2+\z_1)V_{\phi}f_2(z_1+\z_2,z_2)|_{\al\zdd\times \al\zdd}\|_{l^{p,q}_{1\otimes \widetilde{m}}}
    \\
    = &
    \|V_{\phi}f_1(z)V_{\phi}f_2(z+J\z)|_{\al\zdd\times \al\zdd}\|_{l^{p,q}_{1\otimes \widetilde{m}}}
    \\
    = &
    \|(V_{\phi}f_1(\al k_1,\al k_2)V_{\phi}f_2(\al (k_1+n_2),\al(k_2-n_1)))\|_{l^{p,q}_{1\otimes \widetilde{m}}}.
  \end{split}
  \ee
  By the support of $\phi$ and $\va$, we estimate $|V_{\phi}f_1(\al k_1,\al k_2)|$ by
  \be
  \begin{split}
    |V_{\phi}f_1(\al k_1,\al k_2)|
    =
    |\langle f_1, T_{\al k_1}M_{\al k_2}\phi\rangle|
    =&
    \Big|\langle \sum_{j,l}a_{j,l}T_{\al j} M_{\al l}\va, T_{\al k_1}M_{\al k_2}\phi\rangle\Big|
    \\
    = &
    \Big|\langle \sum_{l}a_{k_1,l}T_{\al k_1} M_{\al l}\va, T_{\al k_1}M_{\al k_2}\phi\rangle\Big|
    \\
    = &
     \Big|\sum_{l}a_{k_1,l}\langle M_{\al l}\va, M_{\al k_2}\phi\rangle\Big|
    =
    \Big|\sum_{l}a_{k_1,l}\langle \va, M_{\al(k_2-l)}\phi\rangle\Big|,
  \end{split}
  \ee
  where the interchange of the order of summation and integration is valid since the number of terms is finite.
  Using the nonnegativity of $\vec{a}$, $\vec{b}$,  $\hat{\phi}$ and $\hat{\va}$, we obtain
  \be
  \sum_{l}a_{k_1,l}\langle \va, M_{\al(k_2-l)}\phi\rangle
  =
  \sum_{l}a_{k_1,l}\langle \hat{\va}, T_{\al(k_2-l)}\hat{\phi}\rangle
  \geq a_{k_1,k_2}\langle \hat{\va}, \hat{\phi}\rangle\gtrsim a_{k_1,k_2}.
  \ee
  From the above two estimates, we get the lower bound of $|V_{\phi}f_1(\al k_1,\al k_2)|$:
  \be
  |V_{\phi}f_1(\al k_1,\al k_2)|\gtrsim a_{k_1,k_2}.
  \ee
  A similar calculation yields the lower bound of $|V_{\phi}f_2(\al (k_1+n_2),\al(k_2-n_1))|$:
  \be
  |V_{\phi}f_2(\al (k_1+n_2),\al(k_2-n_1))|\gtrsim b_{k_1+n_2,k_2-n_1}.
  \ee
  By the lower bounds of $|V_{\phi}f_1(\al k_1,\al k_2)|$ and $|V_{\phi}f_2(\al (k_1+n_2),\al(k_2-n_1))|$, we obtain
  \be
  \begin{split}
  \|V_{\P}(W_0(f_1,f_2))|_{\zdd\times \zdd}\|_{l^{p,q}_{1\otimes \widetilde{m}}}
  = &
  \|V_{\phi}f_1(\al k_1,\al k_2)V_{\phi}f_2(\al (k_1+n_2),\al(k_2-n_1))|_{\zdd\times \zdd}\|_{l^{p,q}_{1\otimes \widetilde{m}}}
  \\
  \gtrsim &
  \|(a_{k_1,k_2}b_{k_1+n_2,k_2-n_1})\|_{l^{p,q}_{1\otimes \widetilde{m}}}
  =
  \|a(k)b(k+Jn)\|_{l^{p,q}_{1\otimes \widetilde{m}}(\zdd\times \zdd)}.
  \end{split}
  \ee
  Using this and \eqref{pp-M1-1}, we obtain the lower estimate of $\|W_0(f_1,f_2)\|_{M^{p,q}_{1\otimes m}}$:
  \ben\label{pp-M1-2}
  \|W_0(f_1,f_2)\|_{M^{p,q}_{1\otimes m}}\gtrsim \|a(k)b(k+Jn)\|_{l^{p,q}_{1\otimes \widetilde{m}}(\zdd\times \zdd)}.
  \een

  On the other hand, using the boundedness of synthesis operator in Lemma \ref{lm, bdCD}, we get
  the upper bound estimates of $\|f_i\|_{M^{p_i,q_i}_{m_i}}$:
  \ben\label{pp-M1-3}
  \|f_1\|_{M^{p_1,q_1}_{m_1}}=\|D_{\va}^{\al,\al}\vec{a}\|_{M^{p_1,q_1}_{m_1}}
  \lesssim
  \|\vec{a}\|_{l^{p_1,q_1}_{\widetilde{m_1}}(\zdd)},\ \ \
  \|f_2\|_{M^{p_2,q_2}_{m_2}}=\|D_{\va}^{\al,\al}\vec{b}\|_{M^{p_2,q_2}_{m_2}}
  \lesssim
  \|\vec{b}\|_{l^{p_2,q_2}_{\widetilde{m_2}}(\zdd)}.
  \een
  Combining the estimates \eqref{pp-M1-2} and \eqref{pp-M1-3}, we conclude the desired inequality by
  \be
  \begin{split}
  \|a(k)b(k+Jn)\|_{l^{p,q}_{1\otimes \widetilde{m}}(\zdd\times \zdd)}
  \lesssim
  \|W_0(f_1,f_2)\|_{M^{p,q}_{1\otimes m}}
  \lesssim &
  \|f_1\|_{M^{p_1,q_1}_{m_1}}\|f_2\|_{M^{p_2,q_2}_{m_2}}
  \\
  \lesssim &
  \|\vec{a}\|_{l^{p_1,q_1}_{\widetilde{m_1}}(\zdd)}\|\vec{b}\|_{l^{p_2,q_2}_{\widetilde{m_2}}(\zdd)}.
  \end{split}
  \ee
By a standard limiting argument, the above inequality is valid for all sequences $\vec{a}$ and $\vec{b}$.

\textbf{``If'' part.}
For any fixed $\al>0$, if \eqref{pp-M1-4} is valid, we claim that it also holds for $\frac{\al}{N}$ with any positive integer $N$:
that is, \eqref{pp-M1-4} implies
\be
  \|(a(k)b(k+Jn))\|_{l^{p,q}_{1\otimes \widetilde{\widetilde{m}}}(\zdd\times \zdd)}
\lesssim \|(a(k))\|_{l^{p_1,q_1}_{\widetilde{\widetilde{m_1}}}(\zdd)}\cdot \|(b(k))\|_{l^{p_2,q_2}_{\widetilde{\widetilde{m_2}}}(\zdd)},
\ee
where $\widetilde{\widetilde{m}}(z)=\widetilde{m}(\frac{z}{N})=m(\frac{\al z}{N})$,
$\widetilde{\widetilde{m_i}}(z)=\widetilde{m_i}(\frac{z}{N})=m_i(\frac{\al z}{N})$.
Denote
\be
\La=[0,N)^{4d}\cap \zddd,\ \  \G_{j,l}=(j,l)+N\zddd, (j,l)\in \La.
\ee
We have $\zddd=\bigcup_{(j,l)\in \La}\G_{j,l}$ and
\be
\begin{split}
\|(a(k)b(k+Jn))&\|_{l^{p,q}_{1\otimes \widetilde{\widetilde{m}}}(\zdd\times \zdd)}
\sim
\sum_{(j,l)\in \La}\|(a(k)b(k+Jn))\|_{l^{p,q}_{1\otimes \widetilde{\widetilde{m}}}(\G_{j,l})}
\\
\sim &
\sum_{(j,l)\in \La} \|a(Nk+j)b(N(k+Jn)+(j+Jl))\widetilde{\widetilde{m}}(Nn+l)\|_{l^{p,q}(\zdd\times \zdd)}
\\
\sim &
\sum_{(j,l)\in \La} \|a(Nk+j)b(N(k+Jn)+(j+Jl))\widetilde{m}(n)\|_{l^{p,q}(\zdd\times \zdd)}
\\
= &
\sum_{(j,l)\in \La}\|a(Nk+j)b(N(k+Jn)+(j+Jl))\|_{l^{p,q}_{1\times \widetilde{m}}(\zdd\times \zdd)}.
\end{split}
\ee
Then, we continue this estimate by applying \eqref{pp-M1-4}:
\be
\begin{split}
&\sum_{(j,l)\in \La}\|a(Nk+j)b(N(k+Jn)+(j+Jl))\|_{l^{p,q}_{1\times \widetilde{m}}(\zdd\times \zdd)}
\\
\lesssim &
\sum_{(j,l)\in \La}\|a(Nk+j)\|_{l^{p_1,q_1}_{\widetilde{m_1}}(\zdd)}\|b(Nk+j+Jl)\|_{l^{p_2,q_2}_{\widetilde{m_2}}(\zdd)}
\\
= &
\sum_{(j,l)\in \La}\|a(Nk+j)\widetilde{m_1}(k)\|_{l^{p_1,q_1}(\zdd)}
\|b(Nk+j+Jl)\widetilde{m_2}(k)\|_{l^{p_2,q_2}(\zdd)}
\\
\sim &
\sum_{(j,l)\in \La}\|a(Nk+j)\widetilde{\widetilde{m_1}}(Nk+j)\|_{l^{p_1,q_1}(\zdd)}
\|b(Nk+j+Jl)\widetilde{\widetilde{m_2}}(Nk+j+Jl)\|_{l^{p_2,q_2}(\zdd)}
\\
\lesssim &
\sum_{(j,l)\in \La}\|a(k)\|_{l^{p_1,q_1}_{\widetilde{\widetilde{m_1}}}(\zdd)}
\|b(k)\|_{l^{p_1,q_1}_{\widetilde{\widetilde{m_2}}}(\zdd)}
\lesssim
\|a(k)\|_{l^{p_1,q_1}_{\widetilde{\widetilde{m_1}}}(\zdd)}
\|b(k)\|_{l^{p_2,q_2}_{\widetilde{\widetilde{m_2}}}(\zdd)},
\end{split}
\ee
where in the last inequality we use the fact $|\La|<\fy$.
The claim follows by the above two estimates. From this claim, we find that
if \eqref{pp-M1-4} holds for some $\al>0$, then $\al$ can be taken to be sufficiently small.
Now, we turn to verify the boundedness of $W_0$ under the assumption that \eqref{pp-M1-4} holds for some $\al>0$.

Note that $\Phi=W_0(\phi,\phi)\in \calS(\rdd)$ with $\phi\in \calS(\rd)$.
There exists a sufficiently large integer $N_1$ such that for suitable constants $A, B\in (0,\fy)$
\be
A\leq \sum_{k\in \zdd}|\Phi(x-\frac{\al}{N_1} k)|^2\leq B.
\ee
Denote $\widetilde{\al}=\frac{\al}{N_1}$.
Using Theorem \ref{thm-frame-L^2}, there exists a constant $\b=\widetilde{\al}/N_2=\frac{\al}{N_1N_2}$ with sufficiently large integer $N_2$ such that
$\mathcal {G}(\Phi,\widetilde{\al},\b)$ is a Gabor frame of $L^2(\zdd)$.
Let $\Psi=(S_{\Phi,\Phi}^{\widetilde{\al},\b})^{-1}\Phi$ be the canonical dual widow of $\Phi$.
Note that $\Phi\in \calS\subset \mathfrak{M}^{p,q}_v$, then Definition \ref{df-space-windows} and Theorem \ref{thm-frame-invertible} implies that
$\Psi\in \mathfrak{M}^{p,q}_v$.
By the definitions of $\Phi$ and $\Psi$, we have
$S_{\Phi,\Psi}^{\widetilde{\al},\b}=D_{\Psi}^{\widetilde{\al},\b}C_{\Phi}^{\widetilde{\al},\b}=I$ on $L^2(\rdd)$.
Then Theorem \ref{thm, frame for Mpq} yields that $f=S_{\Phi,\Psi}^{\widetilde{\al},\b}f$ for all $f\in M^{p,q}_{1\otimes m}$.
Recalling $\b=\widetilde{\al}/N_2$ and using Lemma \ref{lm-STFT-tWd}, we find that
\be
\begin{split}
  \|C_{\Phi}^{\widetilde{\al},\b}W_0(f_1,f_2)\|_{l^{p,q}_{1\otimes \widetilde{m}}}
  =
   &
  \|V_{\P}(W_0(f_1,f_2))(z,\z)|_{\widetilde{\al}\zdd\times\b\zdd}\|_{l^{p,q}_{1\otimes \widetilde{m}}}
  \\
  \leq &
  \|V_{\P}(W_0(f_1,f_2))(z,\z)|_{\b\zdd\times\b\zdd}\|_{l^{p,q}_{1\otimes \widetilde{m}}}
  \\
  = &
  \|V_{\phi}f_1(z_1,z_2+\z_1)V_{\phi}f_2(z_1+\z_2,z_2)|_{\b\zdd\times \b\zdd}\|_{l^{p,q}_{1\otimes \widetilde{m}}}
  \\
  = &
  \|V_{\phi}f_1(z)V_{\phi}f_2(z+J\z)|_{\b\zdd\times \b\zdd}\|_{l^{p,q}_{1\otimes \widetilde{m}}}
  \\
  = &
  \left\|V_{\phi}f_1(\b k)V_{\phi}f_2(\b(k+Jn))\right\|_{l^{p,q}_{1\otimes \widetilde{m}}(\zdd\times \zdd)}.
\end{split}
\ee

Using the inequality of discrete mixed-norm spaces \eqref{pp-M1-4} with $\frac{\al}{N_1N_2}$
and Lemma \ref{lm, bdCD}, we continue the above estimate by
\be
\begin{split}
  &\left\|(V_{\phi}f_1(\b k)V_{\phi}f_2(\b(k+Jn)))\right\|_{l^{p,q}_{1\otimes \widetilde{m}}(\zdd\times \zdd)}
  \\
  \lesssim &
  \|(V_{\phi}f_1(\b k))\|_{l^{p_1,q_1}_{\widetilde{m_1}}(\zdd)}
  \cdot \|(V_{\phi}f_2(\b k))\|_{l^{p_2,q_2}_{\widetilde{m_2}}(\zdd)}
  \\
  = &
  \|V_{\phi}f_1(z)|_{\b\zd\times \b\zd}\|_{l^{p_1,q_1}_{\widetilde{m_1}}}
  \|V_{\phi}f_2(z)|_{\b\zd\times \b\zd}\|_{l^{p_2,q_2}_{\widetilde{m_2}}}
  \lesssim
  \|f_1\|_{M^{p_1,q_1}_{m_1}}\cdot \|f_2\|_{M^{p_2,q_2}_{m_2}}.
\end{split}
\ee
The above two estimates imply that
\be
\|C_{\Phi}^{\widetilde{\al},\b}W_0(f_1,f_2)\|_{l^{p,q}_{1\otimes \widetilde{m}}}
\lesssim
\|f_1\|_{M^{p_1,q_1}_{m_1}}\cdot \|f_2\|_{M^{p_2,q_2}_{m_2}}.
\ee
Hence,
\be
\begin{split}
\|W_0(f_1,f_2)\|_{M^{p,q}_{1\otimes m}}
= &\|D_{\Psi}^{\widetilde{\al},\b}C_{\Phi}^{\widetilde{\al},\b}W_0(f_1,f_2)\|_{M^{p,q}_{1\otimes m}}
\\
\lesssim &
\|C_{\Phi}^{\widetilde{\al},\b}W_0(f_1,f_2)\|_{l^{p,q}_{1\otimes \widetilde{m}}}
\lesssim
\|f_1\|_{M^{p_1,q_1}_{m_1}}\cdot \|f_2\|_{M^{p_2,q_2}_{m_2}}.
\end{split}
\ee
We have now completed the proof.
\end{proof}

We are now in a position to give the proof of Theorem \ref{thm-M1}.
As we will see, it follows by Proposition \ref{pp-M1} with some changes of variables.

\begin{proof}[Proof of Theorem \ref{thm-M1}]
Let $\Phi_{\tau}=W_{\tau}(\phi,\phi)$ with nonzero function $\phi\in \calS(\rd)$.
Using Lemma \ref{lm-STFT-tWd}, we obtain that

\be
\begin{split}
  \|W_{\tau}(f_1,f_2)\|_{M^{p,q}_{1\otimes m}}
  = &
  \|V_{\Phi_{\tau}}W_{\tau}(f_1,f_2)\|_{L^{p,q}_{1\otimes m}}
  \\
  = &
  \|V_{\phi}f_1(z_1-\tau \z_2,z_2+(1-\tau)\z_1)V_{\phi}f_2(z_1+(1-\tau)\z_2,z_2-\tau \z_1)\|_{L^{p,q}_{1\otimes m}}
  \\
  = &
  \|V_{\phi}f_1(z_1,z_2+\z_1)V_{\phi}f_2(z_1+\z_2,z_2)\|_{L^{p,q}_{1\otimes m}}
  \\
  = &
  \|V_{\Phi_{0}}W_{0}(f_1,f_2)\|_{L^{p,q}_{1\otimes m}}=\|W_{0}(f_1,f_2)\|_{M^{p,q}_{1\otimes m}}.
\end{split}
\ee
Hence, the boundedness property:
  \be
  W_{\tau}: M^{p_1,q_1}_{m_1}(\rd)\times M^{p_2,q_2}_{m_2}(\rd)\longrightarrow M^{p,q}_{1\otimes m}(\rdd)
  \ee
  is independent with $\tau\in [0,1]$.
We only need to consider the case $\tau=0$.

Write \eqref{pp-M1-4} with $\al=1$ by
\be
\begin{split}
&\bigg(\sum_{n_1,n_2\in \zd}\bigg(\sum_{k_1,k_2\in \zd}|a_{k_1,k_2}b_{k_1+n_2,k_2-n_1}|^p\bigg)^{q/p}m(n_1,n_2)^q\bigg)^{1/q}
\\
\lesssim &
\bigg(\sum_{n_1\in \zd}\bigg(\sum_{k_1\in \zd}|a_{k_1,n_1} m_1(k_1,n_1)|^{p_1}\bigg)^{q_1/p_1}\bigg)^{1/q_1}
\bigg(\sum_{n_2\in \zd}\bigg(\sum_{k_2\in \zd}|b_{k_2,n_2} m_2(k_2,n_2)|^{p_2}\bigg)^{q_2/p_2}\bigg)^{1/q_2},
\end{split}
\ee
which is equivalent to
\be
\begin{split}
&\bigg(\sum_{n_1,n_2\in \zd}\bigg(\sum_{k_1,k_2\in \zd}|a_{k_1,k_2}b_{-n_2-k_1,n_1-k_2}|^p\bigg)^{q/p}m(n_1,n_2)^q\bigg)^{1/q}
\\
\lesssim &
\bigg(\sum_{n_1\in \zd}\bigg(\sum_{k_1\in \zd}|a_{k_1,n_1} m_1(k_1,n_1)|^{p_1}\bigg)^{q_1/p_1}\bigg)^{1/q_1}
\bigg(\sum_{n_2\in \zd}\bigg(\sum_{k_2\in \zd}|b_{-k_2,-n_2} m_2(k_2,n_2)|^{p_2}\bigg)^{q_2/p_2}\bigg)^{1/q_2},
\end{split}
\ee
After some change of variables, one can find the following equivalent form:
\be
\begin{split}
&\bigg(\sum_{n_1,n_2\in \zd}\bigg(\sum_{k_1,k_2\in \zd}|a_{k_1,k_2}b_{n_1-k_1,n_2-k_2}|^p\bigg)^{q/p}m(n_2,-n_1)^q\bigg)^{1/q}
\\
\lesssim &
\bigg(\sum_{n_1\in \zd}\bigg(\sum_{k_1\in \zd}|a_{k_1,n_1} m_1(k_1,n_1)|^{p_1}\bigg)^{q_1/p_1}\bigg)^{1/q_1}
\bigg(\sum_{n_2\in \zd}\bigg(\sum_{k_2\in \zd}|b_{k_2,n_2} m_2(-k_2,-n_2)|^{p_2}\bigg)^{q_2/p_2}\bigg)^{1/q_2}.
\end{split}
\ee
This is equivalent to the desired inequality
\be
\|(a_{k_1,k_2}b_{n_1-k_1,n_2-k_2})\|_{l^{p,q}_{1\otimes m_J}}
\lesssim \|\vec{a}\|_{l^{p_1,q_1}_{m_1}(\zdd)}\|\vec{b}\|_{l^{p_2,q_2}_{\calI m_2}(\zdd)}.
\ee
In particular, when $p<\fy$, this is equivalent to
the convolution inequality:
\be
l^{p_1/p,q_1/p}_{m_1^p}(\zdd)\ast l^{p_2/p,q_2/p}_{\calI m_2^p}(\zdd)\subset l^{q/p,q/p}_{m_J^p}(\zdd).
\ee
Hence, the conclusion in Theorem \ref{thm-M1} (with $\tau=0$) follows directly by Proposition \ref{pp-M1}.
\end{proof}
\begin{remark}
  The reader may observe that, in the proof of Theorem \ref{thm-M1}, we only use Proposition \ref{pp-M1} with $\al=1$.
  However, in order to verify Proposition \ref{pp-M1} with $\al=1$, we actually need \eqref{pp-M1-4} for sufficiently small $\al$.
  Thus, we would like to keep a stronger version with any $\al>0$.
\end{remark}

\subsection{Relations between BMM and BMW for $\tau\in (0,1)$}
\begin{proposition}\label{pp-eqWM}
  Assume $p_i, q_i, p, q \in (0,\fy]$, $i=1,2$. For $\tau\in (0,1)$, we have
  \ben\label{pp-eqWM-1}
  \|W_{\tau}(f_1,f_2)\|_{M^{p,q}_{1\otimes \widetilde{m}_{J^{-1}}}}\lesssim \|f_1\|_{M^{p_1,q_1}_{m_1}}\cdot \|f_2\|_{M^{p_2,q_2}_{\calI \widetilde{m_2}}}
  \een
if and only if
  \ben\label{pp-eqWM-2}
  \|W_{\tau}(f_1,f_2)\|_{W(\scrF L^{{p}},L^{q}_m)}\lesssim \|f_1\|_{M^{p_1,q_1}_{m_1}}\cdot \|f_2\|_{M^{p_2,q_2}_{m_2}},
  \een
  where $\widetilde{m}(\z_1,\z_2)=m((1-\tau)\z_1, \tau \z_2)$, $\widetilde{m_2}(z_1,z_2)=m_2(\frac{1-\tau}{\tau}z_1, \frac{\tau}{1-\tau}z_2)$.
\end{proposition}
\begin{proof}
  Using Lemmas \ref{lm, STFT-twd1}, \ref{lm-ltW} and \ref{lm-chirp}, we obtain
  \ben\label{pp-eqWM-3}
  \begin{split}
  \|W_{\tau}(f_1,f_2)\|_{W(\scrF L^{{p}},L^{q}_m)}
  \sim &
  \|V_{\calD_{\frac{1-\tau}{\tau}}\calI f_2}f_1(\frac{1}{1-\tau}x,\frac{1}{\tau}\xi)\|_{W(\scrF L^{{p}},L^{q}_m)}
  \\
  \sim &
  \|V_{\calD_{\frac{1-\tau}{\tau}}\calI f_2}f_1\|_{W(\scrF L^p,L^q_{\widetilde{m}})},
  \end{split}
  \een
  where $\widetilde{m}(\z_1,\z_2)=m((1-\tau)\z_1, \tau \z_2)$.
  On the other hand, recalling $W(\scrF L^p,L^q_{\widetilde{m}_{J^{-1}}})=\scrF M^{p,q}_{1\otimes \widetilde{m}_{J^{-1}}}$,
  and using Lemma \ref{lm, STFT-twd2} and Lemma \ref{lm-chirp}, we obtain
  \be
  \begin{split}
    \|W_{\tau}(f_1,f_2)\|_{M^{p,q}_{1\otimes \widetilde{m}_{J^{-1}}}}
    = &
    \|\scrF(W_{\tau}(f_1,f_2))\|_{W(\scrF L^p,L^q_{\widetilde{m}_{J^{-1}}})}
    \\
    = &
    \|e^{-2\pi i\tau u_1\cdot u_2}V_{f_2}f_1(-Ju)\|_{W(\scrF L^p,L^q_{\widetilde{m}_{J^{-1}}})}
    \\
    \sim &
    \|V_{f_2}f_1(-Ju)\|_{W(\scrF L^p,L^q_{\widetilde{m}_{J^{-1}}})}.
  \end{split}
  \ee
  Observing that $J^{-1}=-J$,
and applying Lemma \ref{lm-ltW}, we continue the above estimate by
    \ben\label{pp-eqWM-4}
    \begin{split}
    \|W_{\tau}(f_1,f_2)\|_{M^{p,q}_{1\otimes \widetilde{m}_{J^{-1}}}}
    =
    \|V_{f_2}f_1(-Ju)\|_{W(\scrF L^p,L^q_{\widetilde{m}_{J^{-1}}})}
    \sim
    \|V_{f_2}f_1\|_{W(\scrF L^p,L^q_{\widetilde{m}})}.
  \end{split}
  \een
  Using \eqref{pp-eqWM-3} and \eqref{pp-eqWM-4}, we obtain
  \be
  \|W_{\tau}(f_1,f_2)\|_{W(\scrF L^{{p}},L^{q}_m)}
  \sim \|W_{\tau}(f_1,\calD_{\frac{1-\tau}{\tau}}\calI f_2)\|_{M^{p,q}_{1\otimes \widetilde{m}_{J^{-1}}}}.
  \ee
  Moreover, by Lemma \ref{lm-ltM}, we have
  \be
  \|f_2\|_{M^{p_2,q_2}_{m_2}}
  =
  \|\calD_{\frac{\tau}{1-\tau}}\calI\calD_{\frac{1-\tau}{\tau}}\calI f_2\|_{M^{p_2,q_2}_{m_2}}
  \sim
  \|\calD_{\frac{1-\tau}{\tau}}\calI f_2\|_{M^{p_2,q_2}_{\calI\widetilde{m_2}}},
  \ee
  where $\widetilde{m_2}(z_1,z_2)=m(\frac{1-\tau}{\tau}z_1, \frac{\tau}{1-\tau}z_2)$.

 If \eqref{pp-eqWM-1} holds, \eqref{pp-eqWM-2} follows by
 \be
 \begin{split}
   \|W_{\tau}(f_1,f_2)\|_{W(\scrF L^{{p}},L^{q}_m)}
   \sim &
   \|W_{\tau}(f_1,\calD_{\frac{1-\tau}{\tau}}\calI f_2)\|_{M^{p,q}_{1\otimes \widetilde{m}_{J^{-1}}}}
   \\
   \lesssim &
   \|f_1\|_{M^{p_1,q_1}_{m_1}}\cdot \|\calD_{\frac{1-\tau}{\tau}} \calI f_2\|_{M^{p_2,q_2}_{\calI \widetilde{m}_2}}
   \sim
   \|f_1\|_{M^{p_1,q_1}_{m_1}}\cdot \|f_2\|_{M^{p_2,q_2}_{m_2}}.
 \end{split}
 \ee
 Viceversa, if \eqref{pp-eqWM-2} holds, then \eqref{pp-eqWM-1} follows by
 \be
 \begin{split}
   \|W_{\tau}(f_1,f_2)\|_{M^{p,q}_{1\otimes \widetilde{m}_{J^{-1}}}}
   = &
   \|W_{\tau}(f_1,\calD_{\frac{1-\tau}{\tau}}\calI\calD_{\frac{\tau}{1-\tau}}\calI f_2)\|_{M^{p,q}_{1\otimes \widetilde{m}_{J^{-1}}}}
   \\
   \sim &
   \|W_{\tau}(f_1,\calD_{\frac{\tau}{1-\tau}}\calI f_2)\|_{W(\scrF L^{{p}},L^{q}_m)}
   \\
   \lesssim &
   \|f_1\|_{M^{p_1,q_1}_{m_1}}\cdot \|\calD_{\frac{\tau}{1-\tau}}\calI f_2\|_{M^{p_2,q_2}_{m_2}}
   \sim \|f_1\|_{M^{p_1,q_1}_{m_1}}\|f_2\|_{M^{p_2,q_2}_{\calI \widetilde{m_2}}}.
 \end{split}
 \ee
\end{proof}

By the above proposition, we can prove Theorem \ref{thm-W1} for $\tau\in (0,1)$.

\begin{proof}[Proof of Theorem \ref{thm-W1} for $\tau\in (0,1)$]
Observe that $(\widetilde{m}_{J^{-1}})_J=\widetilde{m}$ and $\calI(\calI \widetilde{m_2})=\widetilde{m_2}$. The desired conclusion
follows by Theorem \ref{thm-M1} and Proposition \ref{pp-eqWM}.

\end{proof}

\subsection{Discretization by Gabor coefficients for BMW with endpoints $\tau=0,1$}

\begin{proof}[Proof of Theorem \ref{thm-W1} for $\tau=0$]
We divide the proof into two parts.

\textbf{``Only if'' part.}
   Let $f_1$, $f_2$, $\Phi$ be the same functions in the proof of Theorem \ref{pp-M1}.
Write
\be
\begin{split}
  \|W_0(f_1,f_2)\|_{W(\scrF L^p,L^q_m)}
  = &
  \|\scrF^{-1}(W_0(f_1,f_2))\|_{M^{p,q}_{1\otimes m}}
  \\
  = &
  \|V_{\check{\Phi}}(\scrF^{-1}(W_0(f_1,f_2)))(z,\z)\|_{L^{p,q}_{1\otimes m}}
  \\
  = &
  \|V_{\Phi}(W_0(f_1,f_2))(\z,-z)\|_{L^{p,q}_{1\otimes m}}
  \\
  = &
  \|V_{\phi}f_1(\z_1,\z_2-z_1)V_{\phi}f_2(\z_1-z_2,\z_2)\|_{L^{p,q}_{1\otimes m}}
  \\
  = &
  \|V_{\phi}f_1(\z_1,z_1)V_{\phi}f_2(z_2,\z_2)\|_{L^{p,q}_{1\otimes m}}.
\end{split}
\ee
By the same argument in the proof of Proposition \ref{pp-M1}, we have the lower estimate:
\be
  \begin{split}
  \|W_0(f_1,f_2)\|_{W(\scrF L^p,L^q_m)}
  \gtrsim &
  \|V_{\phi}f_1(n_1,k_1)V_{\phi}f_2(k_2,n_2)|_{\zdd\times \zdd}\|_{l^{p,q}_{1\otimes m}}
  \gtrsim
  \|(a_{n_1,k_1}b_{k_2,n_2})\|_{l^{p,q}_{1\otimes m}},
  \end{split}
  \ee
  and the upper bound estimates
  $  \|f_1\|_{M^{p_1,q_1}_{m_1}}
  \lesssim
  \|\vec{a}\|_{l^{p_1,q_1}_{m_1}(\zdd)},
  \|f_2\|_{M^{p_2,q_2}_{m_2}}
  \lesssim
  \|\vec{b}\|_{l^{p_2,q_2}_{m_2}(\zdd)}$.
  Combining the above two estimates, we deduce that
  \be
  \begin{split}
  \|(a_{n_1,k_1}b_{k_2,n_2})\|_{l^{p,q}_{1\otimes m}}
  \lesssim
  \|W_0(f_1,f_2)\|_{M^{p,q}_{1\otimes m}}
  \lesssim &
  \|f_1\|_{M^{p_1,q_1}_{m_1}}\|f_2\|_{M^{p_2,q_2}_{m_2}}
  \\
  \lesssim &
  \|\vec{a}\|_{l^{p_1,q_1}_{m_1}(\zdd)}\|\vec{b}\|_{l^{p_2,q_2}_{m_2}(\zdd)}.
  \end{split}
  \ee
  Note that, by the same method in the proof of Proposition \ref{pp-M1}, the above inequality
  is also valid if $m$,$m_i$ are replaced by $\widetilde{m}$ and $\widetilde{m_i}$ respectively.
  Here $\widetilde{m}(z):=m(\frac{z}{N})$, and $\widetilde{m_i}(z):=m_i(\frac{z}{N})$, $z\in \rdd$, $N\in \mathbb{N}$,
  $i=1,2$.

  \textbf{``If'' part.}
  As in the proof of Theorem \ref{pp-M1}, for $\Phi\in \calS(\rdd)$ mentioned above,
  there exists $\widetilde{\Psi}\in \mathfrak{M}^{p,q}_v$, such that
  $S_{\check{\Phi},\widetilde{\Psi}}^{\al,\b}=D_{\widetilde{\Psi}}^{\al,\b}C_{\check{\Phi}}^{\al,\b}=I$ on $M^{p,q}_{1\otimes m}$, where $\al=1/N_1$, $\b=\al/N_2$,  with some large integers $N_1, N_2$.

Applying Lemma \ref{lm-STFT-tWd}, we find that
\be
\begin{split}
  \|C_{\check{\Phi}}^{\al,\b}\scrF^{-1}W_0(f_1,f_2)\|_{l^{p,q}_{1\otimes \widetilde{m}}}
  = &
  \|V_{\check{\P}}(\scrF^{-1}W_0(f_1,f_2))(z,\z)|_{\al\zdd\times\b\zdd}\|_{l^{p,q}_{1\otimes \widetilde{m}}}
  \\
  \leq &
  \|V_{\check{\P}}(\scrF^{-1}W_0(f_1,f_2))(z,\z)|_{\b\zdd\times\b\zdd}\|_{l^{p,q}_{1\otimes \widetilde{m}}}
  \\
  = &
  \|V_{\phi}f_1(\z_1,z_1)V_{\phi}f_2(z_2,\z_2)|_{\b\zdd\times \b\zdd}\|_{l^{p,q}_{1\otimes \widetilde{m}}}
  \\
  \lesssim &
  \left\|V_{\phi}f_1(\b (n_1,k_1))V_{\phi}f_2(\b(k_2,n_2))\right\|_{l^{p,q}_{1\otimes \widetilde{m}}(\zdd\times \zdd)}.
\end{split}
\ee

Using the inequality of discrete mixed-norm spaces and the sampling property of STFT, we continue the above estimate by
\be
\begin{split}
  &\left\|V_{\phi}f_1(\b (n_1,k_1))V_{\phi}f_2(\b(k_2,n_2))\right\|_{l^{p,q}_{1\otimes \widetilde{m}}(\zdd\times \zdd)}
  \\
  \lesssim &
  \|(V_{\phi}f_1(\b k))\|_{l^{p_1,q_1}_{\widetilde{m_1}}(\zdd)}
  \cdot \|(V_{\phi}f_2(\b k))\|_{l^{p_2,q_2}_{\widetilde{m_2}}(\zdd)}
  \\
  = &
  \|V_{\phi}f_1(z)|_{\b\zdd\times \b\zdd}\|_{l^{p_1,q_1}_{\widetilde{m_1}}}
  \|V_{\phi}f_2(z)|_{\b\zdd\times \b\zdd}\|_{l^{p_2,q_2}_{\widetilde{m_2}}}
  \lesssim
  \|f_1\|_{M^{p_1,q_1}_{m_1}}\cdot \|f_2\|_{M^{p_2,q_2}_{m_2}}.
\end{split}
\ee
The above two estimates imply that
\be
\|C_{\check{\Phi}}^{\al,\b}\scrF^{-1}W_0(f_1,f_2)\|_{l^{p,q}_{1\otimes \widetilde{m}}}
\lesssim
\|f_1\|_{M^{p_1,q_1}_{m_1}}\cdot \|f_2\|_{M^{p_2,q_2}_{m_2}}.
\ee
Hence,
\be
\begin{split}
\|W_0(f_1,f_2)\|_{W(\scrF L^p,L^q_m)}
= &
\|\scrF^{-1}W_0(f_1,f_2)\|_{M^{p,q}_{1\otimes m}}
\\
= &\|D_{\widetilde{\Psi}}^{\al,\b}C_{\check{\Phi}}^{\al,\b}\scrF^{-1}W_0(f_1,f_2)\|_{M^{p,q}_{1\otimes m}}
\\
\lesssim &
\|C_{\check{\Phi}}^{\al,\b}\scrF^{-1}W_0(f_1,f_2)\|_{l^{p,q}_{1\otimes \widetilde{m}}}
\lesssim
\|f_1\|_{M^{p_1,q_1}_{m_1}}\cdot \|f_2\|_{M^{p_2,q_2}_{m_2}}.
\end{split}
\ee
We have now completed the proof.
\end{proof}

\begin{proof}[Proof of Theorem \ref{thm-W1} for $\tau=1$]
A direct calculation yields that
\be
\begin{split}
  \|W_1(f_1,f_2)\|_{W(\scrF L^p,L^q_m)}
  = &
  \|\overline{W_0(f_2,f_1)}\|_{W(\scrF L^p,L^q_m)}
  \\
  = &
  \|V_{\bar{\Phi}}(\overline{W_0(f_2,f_1)})(\z,-z)\|_{L^{p,q}_{1\otimes m}}
  \\
  = &
  \|V_{\Phi}(W_0(f_2,f_1))(\z,z)\|_{L^{p,q}_{1\otimes m}}
  \\
  = &
  \|V_{\Phi}(W_0(f_2,f_1))(\z,-z)\|_{L^{p,q}_{1\otimes m}}
  \\
  = &
  \|W_0(f_2,f_1)\|_{W(\scrF L^p,L^q_m)}.
\end{split}
\ee
The desired result follows by this and the case of $\tau=0$.
\end{proof}

\section{Second characterizations: separations of mixed-norm inequalities }
\subsection{Separation of  mixed-norm convolution inequality}
\begin{proposition}[Separation of convolution]\label{pp-sepc}
Assume $p_i, q_i, p, q \in (0,\fy]$, $i=1,2$. Suppose that $m$ is  submultiplicative,
$m_i=\om_i\otimes \mu_i$, $i=1,2$. We have
\ben\label{pp-sepc-c0}
\|(a_{k_1,k_2}b_{n_1-k_1,n_2-k_2})\|_{l^{p,q}_{1\otimes m}(\zddd)}
\lesssim \|\vec{a}\|_{l^{p_1,q_1}_{m_1}(\zdd)}\|\vec{b}\|_{l^{p_2,q_2}_{m_2}(\zdd)}
\een
if and only if
\begin{eqnarray}
\label{pp-sepc-c1}
&l_{\om_1^p}^{p_1/p}(\zd)\ast l_{\om_2^p}^{p_2/p}(\zd)\subset  l_{m_{\al}^p}^{q/p}(\zd),\ \ \
l_{\mu_1^p}^{q_1/p}(\zd)\ast l_{\mu_2^p}^{q_2/p}(\zd)\subset  l_{m_{\b}^p}^{q/p}(\zd),\ \ \
&p<\fy,
\\
\label{pp-sepc-c2}
&l^{p_1}_{\om_1}(\zd),\ l^{p_2}_{\om_2}(\zd)\subset l^q_{m_{\al}}(\zd),\ \ \ \
l^{q_1}_{\mu_1}(\zd),\ l^{q_2}_{\mu_2}(\zd)\subset l^q_{m_{\b}}(\zd),
&p\geq q.
\end{eqnarray}
Here, we denote $m_{\al}(z_1)=m(z_1,0)$ and $m_{\b}(z_2)=m(0,z_2)$ for $z_1,z_2\in \rd$.
\end{proposition}
\begin{proof}
We divide the proof into two parts.

\textbf{``Only if'' part.}
 In this part, we separate \eqref{pp-sepc-c0} by means of testing it by several constructed discrete sequences.

\textbf{Test 1.} Set $a_{k_1,n_1}=0$ if $n_1\neq 0_{\zd}$, and $b_{k_2,n_2}=0$ if $n_2\neq 0_{\zd}$.
The inequality \eqref{pp-sepc-c0} says that
\ben\label{pp-sepc-t1}
\begin{split}
&\bigg(\sum_{n_1\in \zd}\bigg(\sum_{k_1\in \zd}|a_{k_1,0}b_{n_1-k_1,0}|^p\bigg)^{q/p} m(n_1,0)^q\bigg)^{1/q}
\\
\lesssim &
\bigg(\sum_{k_1\in \zd}|a_{k_1,0} m_1(k_1,0)|^{p_1}\bigg)^{1/p_1}
\bigg(\sum_{k_2\in \zd}|b_{k_2,0} m_2(k_2,0)|^{p_2}\bigg)^{1/p_2}.
\end{split}
\een
Note that for $p<\fy$ the above inequality is equivalent to
$l_{\om_1^p}^{p_1/p}(\zd)\ast l_{\om_2^p}^{p_2/p}(\zd)\subset  l_{m_{\al}^p}^{q/p}(\zd)$.
Moreover, if we further assume $b_{k_2,n_2}=0$ for all $n_2, k_2\neq 0_{\zd}$,
\eqref{pp-sepc-t1} implies the embedding relation
$l^{p_1}_{\om_1}(\zd)\subset l^q_{m_{\al}}(\zd)$.
On the other hand, by taking $a_{k_1,n_1}=0$ for all $k_1, n_1\neq 0_{\zd}$,
we get $l^{p_2}_{\om_2}(\zd)\subset l^q_{m_{\al}}(\zd)$.

\textbf{Test 2.} Set $a_{k_1,n_1}=0$ if $k_1\neq 0_{\zd}$, and $b_{k_2,n_2}=0$ if $k_2\neq 0_{\zd}$.
The inequality \eqref{pp-sepc-c0} says that
\ben\label{pp-sepc-t2}
\begin{split}
&\bigg(\sum_{n_2\in \zd}\bigg(\sum_{k_2\in \zd}|a_{0,k_2}b_{0,n_2-k_2}|^p\bigg)^{q/p} m(0,n_2)^q\bigg)^{1/q}
\\
\lesssim &
\bigg(\sum_{n_1\in \zd}|a_{0,n_1} m_1(0,n_1)|^{q_1}\bigg)^{1/q_1}
\bigg(\sum_{n_2\in \zd}|b_{0,n_2} m_2(0,n_2)|^{q_2}\bigg)^{1/q_2}.
\end{split}
\een
If $p<\fy$, the above inequality is equivalent to
$l_{\mu_1^p}^{q_1/p}(\zd)\ast l_{\mu_2^p}^{q_2/p}(\zd)\subset  l_{m_{\b}^p}^{q/p}(\zd)$.
By taking $b_{k_2,n_2}=0$ for all $n_2, k_2\neq 0_{\zd}$,
\eqref{pp-sepc-t1} becomes the embedding relation
$l^{q_1}_{\mu_1}(\zd)\subset l^q_{m_{\b}}(\zd)$.
Similarly, by taking $a_{k_1,n_1}=0$ for all $k_1, n_1\neq 0_{\zd}$,
we get $l^{q_2}_{\mu_2}(\zd)\subset l^q_{m_{\b}}(\zd)$.

From Tests 1 and 2, we get \eqref{pp-sepc-c1} and \eqref{pp-sepc-c2}.


\textbf{``If'' part.} In this part, we consider following cases.

\textbf{Case 1: $p<q$.}
In this case, we will verify \eqref{pp-sepc-c1} implies \eqref{pp-sepc-c0}.
Using the Minkowski inequality and $m(n_1,n_2)\lesssim m_{\al}(n_1)m_{\b}(n_2)$, we conclude that
\be
\begin{split}
&\bigg(\sum_{n_1,n_2\in \zd}\bigg(\sum_{k_1,k_2\in \zd}|a_{k_1,k_2}b_{n_1-k_1,n_2-k_2}|^p\bigg)^{q/p} m(n_1,n_2)^q\bigg)^{1/q}
\\
\lesssim &
\bigg(\sum_{n_1,n_2\in \zd}\bigg(\sum_{k_1,k_2\in \zd}|a_{k_1,k_2}b_{n_1-k_1,n_2-k_2}|^p\bigg)^{q/p} m_{\al}(n_1)^qm_{\b}(n_2)^q\bigg)^{1/q}
\\
\lesssim &
\bigg(\sum_{n_2\in \zd}\bigg(\sum_{k_2\in \zd}\bigg(\sum_{n_1\in \zd}\bigg(\sum_{k_1\in \zd}|a_{k_1,k_2}b_{n_1-k_1,n_2-k_2}|^p\bigg)^{q/p} m_{\al}(n_1)^q\bigg)^{p/q}\bigg)^{q/p} m_{\b}(n_2)^q\bigg)^{1/q}.
\end{split}
\ee
Using the convolution inequality $l_{\om_1^p}^{p_1/p}(\zd)\ast l_{\om_2^p}^{p_2/p}(\zd)\subset  l_{m_{\al}^p}^{q/p}(\zd)$,
we continue the above estimate by
\be
\begin{split}
  &\bigg(\sum_{n_2\in \zd}\bigg(\sum_{k_2\in \zd}\bigg(\sum_{n_1\in \zd}\bigg(\sum_{k_1\in \zd}|a_{k_1,k_2}b_{n_1-k_1,n_2-k_2}|^p\bigg)^{q/p} m_{\al}(n_1)^q\bigg)^{p/q}\bigg)^{q/p} m_{\b}(n_2)^q\bigg)^{1/q}
  \\
  \lesssim &
  \bigg(\sum_{n_2\in \zd}\bigg(\sum_{k_2\in \zd}
  \|(a_{k_1,k_2})_{k_1}\|_{l^{p_1}_{\om_1}}^p\|(b_{n_1,n_2-k_2})_{n_1}\|_{l^{p_2}_{\om_2}}^p \bigg)^{q/p} m_{\b}(n_2)^q\bigg)^{1/q}.
\end{split}
\ee
Then, the convolution inequality $l_{\mu_1^p}^{q_1/p}(\zd)\ast l_{\mu_2^p}^{q_2/p}(\zd)\subset  l_{m_{\b}^p}^{q/p}(\zd)$ further implies that
\be
\begin{split}
  &\bigg(\sum_{n_2\in \zd}\bigg(\sum_{k_2\in \zd}
  \|(a_{k_1,k_2})_{k_1}\|_{l^{p_1}_{\om_1}}^p\|(b_{n_1,n_2-k_2})_{n_1}\|_{l^{p_2}_{\om_2}}^p \bigg)^{q/p} m_{\b}(n_2)^q\bigg)^{1/q}
  \\
  \lesssim &
  \|(\|(a_{k_1,n_1})_{k_1}\|_{l^{p_1}_{\om_1}})_{n_1}\|_{l^{q_1}_{\mu_1}}
  \|(\|(b_{k_2,n_2})_{k_2}\|_{l^{p_2}_{\om_2}})_{n_2}\|_{l^{q_2}_{\mu_2}}
  =\|\vec{a}\|_{l^{p_1,q_1}_{\om_1\otimes \mu_1}}\|\vec{b}\|_{l^{p_2,q_2}_{\om_2\otimes \mu_2}}.
\end{split}
\ee

\textbf{Case 2: $p\geq q$.}
In this case, we will verify \eqref{pp-sepc-c2} implies \eqref{pp-sepc-c0},
then the conclusion $\eqref{pp-sepc-c1}\Longrightarrow \eqref{pp-sepc-c0}$ follows by the fact
that \eqref{pp-sepc-c1} implies \eqref{pp-sepc-c2} for $p<\fy$.
By the well known embedding relation $l^q\subset l^p$ for $p\geq q$,
and the submultiplicative property of $m$, we have
\be
\begin{split}
  &\bigg(\sum_{n_1,n_2\in \zd}\bigg(\sum_{k_1,k_2\in \zd}|a_{k_1,k_2}b_{n_1-k_1,n_2-k_2}|^p\bigg)^{q/p} m(n_1,n_2)^q\bigg)^{1/q}
\\
\lesssim &
  \bigg(\sum_{n_1,n_2,k_1,k_2\in \zd}|a_{k_1,k_2}|^q|b_{n_1-k_1,n_2-k_2}|^qm(k_1,k_2)^qm(n_1-k_1,n_2-k_2)^q\bigg)^{1/q}
  \\
  = &
  \bigg(\sum_{k_1,k_2\in \zd}|a_{k_1,k_2}|^qm(k_1,k_2)^q\bigg)^{1/q}
  \bigg(\sum_{n_1,n_2\in \zd}|b_{n_1-k_1,n_2-k_2}|^qm(n_1-k_1,n_2-k_2)^q\bigg)^{1/q}
  \\
  \lesssim &
  \|\vec{a}\|_{l^{q,q}_{m_{\al}\otimes m_{\b}}}\|\vec{b}\|_{l^{q,q}_{m_{\al}\otimes m_{\b}}}.
\end{split}
\ee
Using \eqref{pp-sepc-c2}, i.e., the embedding relations
$l^{p_1}_{\om_1}(\zd),\ l^{p_2}_{\om_2}(\zd)\subset l^q_{m_{\al}}(\zd)$ and
$l^{q_1}_{\mu_1}(\zd),\ l^{q_2}_{\mu_2}(\zd)\subset l^q_{m_{\b}}(\zd)$,
 we obtain the following embedding relations for discrete mixed-norm spaces:
\be
l^{p_1,q_1}_{\om_1\otimes \mu_1}(\zdd)\subset l^{q,q}_{m_{\al}\otimes m_{\b}}(\zdd),
\ \ \ \
l^{p_2,q_2}_{\om_2\otimes \mu_2}(\zdd)\subset l^{q,q}_{m_{\al}\otimes m_{\b}}(\zdd).
\ee
Now, we continue our main estimate by
\be
\begin{split}
  &\bigg(\sum_{n_1,n_2\in \zd}\bigg(\sum_{k_1,k_2\in \zd}|a_{k_1,k_2}b_{n_1-k_1,n_2-k_2}|^p\bigg)^{q/p} m(n_1,n_2)^q\bigg)^{1/q}
\\
\lesssim &
\|\vec{a}\|_{l^{q,q}_{m_{\al}\otimes m_{\b}}}\|\vec{b}\|_{l^{q,q}_{m_{\al}\otimes m_{\b}}}
\lesssim
\|\vec{a}\|_{l^{p_1,q_1}_{\om_1\otimes \mu_1}}\|\vec{b}\|_{l^{p_2,q_2}_{\om_2\otimes \mu_2}}.
\end{split}
\ee
This concludes the proof.
\end{proof}

\begin{proof}[Proof of Theorem \ref{thm-M2}]
This proof follows directly by Theorem \ref{thm-M1} and Proposition \ref{pp-sepc}
with the fact that $(m_J)_{\al}=\calI m_{\b}$ and $(m_J)_{\b}=m_{\al}$.
\end{proof}

\subsection{Separation of  mixed-norm embedding inequality}
\begin{proposition}[Separation of embedding]\label{pp-sepe}
Assume $p_i, q_i, p, q \in (0,\fy]$, $i=1,2$. Suppose that $m$ is  submultiplicative, $m_i=\om_i\otimes \mu_i$, $i=1,2$. We have
\ben\label{pp-sepe-c0}
\|(a_{n_1,k_1}b_{k_2,n_2})\|_{l^{p,q}_{1\otimes m}}\lesssim \|\vec{a}\|_{l^{p_1,q_1}_{m_1}(\zdd)}\|\vec{b}\|_{l^{p_2,q_2}_{m_2}(\zdd)}
\een
if and only if the following two convolution relations:
\ben\label{pp-sepe-c1}
\ l^{p_2}_{\om_2}\subset l^p,\ l^{q_2}_{\mu_2}\subset l^q_{m_{\b}},
\een
\ben\label{pp-sepe-c2}
l^{p_1,q_1}_{\om_1\otimes \mu_1}\subset l^{(q,p)}_{m_{\al}\otimes 1}
\een
hold.
Here, we denote $m_{\al}(z_1)=m(z_1,0)$ and $m_{\b}(z_2)=m(0,z_2)$ for $z_1,z_2\in \rd$.
\end{proposition}
\begin{proof}
We divide the proof into two parts.

\textbf{``Only if'' part.}
  Write \eqref{pp-sepe-c0} by
\ben\label{pp-sepe-1}
\begin{split}
&\bigg(\sum_{n_1,n_2\in \zd}\bigg(\sum_{k_1,k_2\in \zd}|a_{n_1,k_1}b_{k_2,n_2}|^p\bigg)^{q/p} m(n_1,n_2)^q\bigg)^{1/q}
\\
\lesssim &
\bigg(\sum_{n_1\in \zd}\bigg(\sum_{k_1\in \zd}|a_{k_1,n_1} m_1(k_1,n_1)|^{p_1}\bigg)^{q_1/p_1}\bigg)^{1/q_1}
\bigg(\sum_{n_2\in \zd}\bigg(\sum_{k_2\in \zd}|b_{k_2,n_2} m_2(k_2,n_2)|^{p_2}\bigg)^{q_2/p_2}\bigg)^{1/q_2}.
\end{split}
\een
The above inequality will be tested by several constructed discrete sequences.

\textbf{Test 1.} Set $a_{k_1,n_1}=0$ if $(k_1,n_1)\neq 0_{\zd\times \zd}$, and $b_{k_2,n_2}=0$ if $n_2\neq 0_{\zd}$.
The inequality \eqref{pp-sepe-1} says that
\be
\begin{split}
\bigg(\sum_{k_2\in \zd}|b_{k_2,0}|^p\bigg)^{1/p}
\lesssim
\bigg(\sum_{k_2\in \zd}|b_{k_2,0} m_2(k_2,0)|^{p_2}\bigg)^{1/p_2},
\end{split}
\ee
which implies the embedding relation $l^{p_2}_{\om_2}\subset l^p$ in \eqref{pp-sepe-c1}.

\textbf{Test 2.} Set $a_{k_1,n_1}=0$ if $(k_1,n_1)\neq 0_{\zd\times \zd}$, and $b_{k_2,n_2}=0$ if $k_2\neq 0_{\zd}$.
The inequality \eqref{pp-sepe-1} says that
\be
\begin{split}
\bigg(\sum_{n_1\in \zd}|b_{0,n_2}|^q m(0,n_2)^q\bigg)^{1/q}
\lesssim
\bigg(\sum_{n_2\in \zd}|b_{0,n_2} m_2(0,n_2)|^{q_2}\bigg)^{1/q_2}.
\end{split}
\ee
This is just the embedding relation $l_{\mu_2}^{q_2}\subset  l_{m_{\b}}^{q}$ in \eqref{pp-sepe-c1}.

\textbf{Test 3.}
Set $b_{k_2,n_2}=0$ if $(k_2,n_2)\neq 0_{\zd\times \zd}$.
The inequality \eqref{pp-sepe-1} says that
\be
\begin{split}
\bigg(\sum_{n_1\in \zd}\bigg(\sum_{k_1\in \zd}|a_{n_1,k_1}|^p\bigg)^{q/p} m(n_1,0)^q\bigg)^{1/q}
\lesssim
\bigg(\sum_{n_1\in \zd}\bigg(\sum_{k_1\in \zd}|a_{k_1,n_1} m_1(k_1,n_1)|^{p_1}\bigg)^{q_1/p_1}\bigg)^{1/q_1}.
\end{split}
\ee
This is just the embedding relation $l^{p_1,q_1}_{\om_1\otimes \mu_1}\subset l^{(q,p)}_{m_{\al}\otimes 1}$ in \eqref{pp-sepe-c2}.
\\

\textbf{``If'' part.} Recall $m(n_1,n_2)\lesssim m_{\al}(n_1)m_{\b}(n_2)$.
Write
\be
\begin{split}
&\bigg(\sum_{n_1,n_2\in \zd}\bigg(\sum_{k_1,k_2\in \zd}|a_{n_1,k_1}b_{k_2,n_2}|^p\bigg)^{q/p} m(n_1,n_2)^q\bigg)^{1/q}
\\
\lesssim &
\bigg(\sum_{n_1,n_2\in \zd}\bigg(\sum_{k_1\in \zd}|a_{n_1,k_1}|^p\sum_{k_2\in \zd}|b_{k_2,n_2}|^p\bigg)^{q/p}
m_{\al}(n_1)^qm_{\b}(n_2)^q\bigg)^{1/q}
\\
\lesssim &
\bigg(\sum_{n_1\in \zd}\bigg(\sum_{k_1\in \zd}|a_{n_1,k_1}|^p\bigg)^{q/p} m_{\al}(n_1)^q\bigg)^{1/q}
\bigg(\sum_{n_2\in \zd}\bigg(\sum_{k_2\in \zd}|b_{k_2,n_2}|^{p}\bigg)^{q/p}m_{\b}(n_2)^q\bigg)^{1/q}
\\
\lesssim &
\bigg(\sum_{n_1\in \zd}\bigg(\sum_{k_1\in \zd}|a_{k_1,n_1} m_1(k_1,n_1)|^{p_1}\bigg)^{q_1/p_1}\bigg)^{1/q_1}
\bigg(\sum_{n_2\in \zd}\bigg(\sum_{k_2\in \zd}|b_{k_2,n_2} m_2(k_2,n_2)|^{p_2}\bigg)^{q_2/p_2}\bigg)^{1/q_2},
\end{split}
\ee
where in the last inequality we use \eqref{pp-sepe-c1} and \eqref{pp-sepe-c2}.
\end{proof}
\begin{proof}[Proof of Theorem \ref{thm-W2}]
Note that $\widetilde{m}$ is submultiplicative if $m$ is submultiplicative.
Observe that $\widetilde{m_2}=\widetilde{\om_2}\otimes \widetilde{\mu_2}$.
The case $\tau\in (0,1)$ follows directly by Theorem \ref{thm-W1} and Proposition \ref{pp-sepc}.
The endpoint cases $\tau=0,1$ follows by Theorem \ref{thm-W1} and Proposition \ref{pp-sepe}.
\end{proof}

\section{Third characterizations: applications for power weights}
\subsection{Sharp exponents for convolution inequalities}

Note that Theorem \ref{thm-M3} is a direct conclusion of Theorem \ref{thm-M2}.
Observe that in Theorem \ref{thm-W2}, $\widetilde{m}\sim m$,
$\widetilde{\om_2}\sim \om_2$ and $\widetilde{\mu_2}\sim \mu_2$,
for $m=v_s$, $\om_2=v_{s_2}$, $\mu_2=v_{t_2}$.
Then Theorem \ref{thm-W3} with $\tau\in (0,1)$ follows by Theorem \ref{thm-W2}.
See the proof of Theorem \ref{thm-W3} for $\tau=0,1$ in the next subsection.

If we want to get the sharp exponents for the convolution inequalities mentioned in Theorem \ref{thm-M3} and \ref{thm-W3},
the following two lemmas is needed.

\begin{lemma}(See \cite[Theorem 1.1]{GuoFanWuZhao2018Studia}) \label{lm-exp-wcov}
Suppose $1\leq q, q_1, q_2\leq \infty$, $s, s_1, s_2\in \mathbb{R}$. Then
\begin{equation}
l^{q_1}_{s_1}(\rd)\ast l^{q_2}_{s_2}(\rd) \subset l^{q}_{s}(\rd)
\end{equation}
if and only if $(\mathbf{q},\mathbf{s})=(q, q_1, q_2, s, s_1, s_2)$ satisfies one of the following conditions $\mathcal {A}_i$, $i=1,2,3,4$.
\begin{eqnarray}
&&(\mathcal {A}_1)\begin{cases}
s\leq s_1,~s\leq s_2,~0\leq s_1+s_2,\\
1+\Big(\frac{1}{q}+\frac{s}{d}\Big)\vee 0<\Big(\frac{1}{q_1}+\frac{s_1}{d}\Big)\vee 0+\Big(\frac{1}{q_2}+\frac{s_2}{d}\Big)\vee 0,\\
\frac{1}{q}+\frac{s}{d}\leq \frac{1}{q_1}+\frac{s_1}{d},~\frac{1}{q}+\frac{s}{d}\leq \frac{1}{q_2}+\frac{s_2}{d},
1\leq \frac{1}{q_1}+\frac{s_1}{d}+\frac{1}{q_2}+\frac{s_2}{d},\\
(q,s)=(q_1,s_1) ~\text{if}~ \frac{1}{q}+\frac{s}{d}=\frac{1}{q_1}+\frac{s_1}{d}, \\
(q,s)=(q_2,s_2) ~\text{if}~ \frac{1}{q}+\frac{s}{d}=\frac{1}{q_2}+\frac{s_2}{d},\\
(q'_1, -s_1)=(q_2,s_2)~\text{if}~ 1=\frac{1}{q_1}+\frac{s_1}{d}+\frac{1}{q_2}+\frac{s_2}{d};
\end{cases}
\\
&&(\mathcal {A}_2)\begin{cases}
s=s_1=s_2=0,\\
q=q_1, q_2=1~or~q=q_2, q_1=1~or~q=\infty, \frac{1}{q_1}+\frac{1}{q_2}=1;
\end{cases}
\\
&&(\mathcal {A}_3)\begin{cases}
s\leq s_1,~s\leq s_2,\\
\frac{1}{q_1}+\frac{1}{q_2}=1,~s_1+s_2=0,\\
\frac{1}{q}+\frac{s}{d}<0\leq \frac{1}{q_1}+\frac{s_1}{d},\frac{1}{q_2}+\frac{s_2}{d};
\end{cases}
\end{eqnarray}
\begin{eqnarray}
&&(\mathcal {A}_4)\begin{cases}
s\leq s_1,~s\leq s_2,~0\leq s_1+s_2,\\
1+\frac{1}{q}+\frac{s}{d}=\frac{1}{q_1}+\frac{s_1}{d}+\frac{1}{q_2}+\frac{s_2}{d},~\frac{1}{q}\leq \frac{1}{q_1}+\frac{1}{q_2},\\
\frac{1}{q}+\frac{s}{d}<\frac{1}{q_1}+\frac{s_1}{d},~\frac{1}{q}+\frac{s}{d}<\frac{1}{q_2}+\frac{s_2}{d},~\frac{1}{q}+\frac{s}{d}>0,\\
q\neq \infty,~q_1,q_2\neq 1,~\text{if}~s=s_1~or~s=s_2.
\end{cases}
\end{eqnarray}
Here, we use the notation \
\begin{equation*}
a\vee b=\max \{a,b\}.
\end{equation*}
\end{lemma}

\begin{lemma}(see \cite[Proposition 2.5]{GuoFanWuZhao2018Studia}) \label{lm-exp-cov}
Suppose $0<q,q_1,q_2\leq \infty$. Then
\ben\label{lm-exp-cov-c0}
l^{q_1}(\rd)\ast l^{q_2}(\rd) \subset l^{q}(\rd)
\een
holds if and only if
\ben\label{lm-exp-cov-c1}
1+\frac{1}{q}\leq\frac{1}{q_1}+\frac{1}{q_2},\
\frac{1}{q}\leq \frac{1}{q_1},\ \frac{1}{q}\leq \frac{1}{q_2}.
\een
Moreover, if \eqref{lm-exp-cov-c1} holds, we have
\ben\label{lm-exp-cov-c2}
l_{|s|}^{q_1}(\rd)\ast l_s^{q_2}(\rd) \subset l_s^{q}(\rd).
\een
\end{lemma}
\begin{proof}
  We only point out that \eqref{lm-exp-cov-c2} is a direct conclusion of \eqref{lm-exp-cov-c0}
  and $\langle j\rangle^s\lesssim \langle j-l\rangle^{|s|}\langle l\rangle^s$ for all $s\in \rr$.
\end{proof}

By the above two lemmas, we obtain the following results.

\begin{theorem}\label{thm-M4}
Suppose that $p<\fy$ and $p\leq p_i,q_i,q$ for $i=1,2$, $\tau\in [0,1]$. We have
\be
W_{\tau}:  \calM^{p_1,q_1}_{v_{s_1,t_1}}(\rd)\times \calM^{p_2,q_2}_{v_{s_2,t_2}}(\rd)
\longrightarrow M^{p,q}_{1\otimes v_s}(\rdd)
\ee
if and only if
\be
(q/p,p_1/p,p_2/p,ps,ps_1,ps_2),\ (q/p,q_1/p,q_2/p,ps,pt_1,pt_2)\in \calA.
\ee
\end{theorem}
\begin{proof}
  By Theorem \ref{thm-M3}, we have
  \be
  W_{\tau}:  \calM^{p_1,q_1}_{v_{s_1,t_1}}(\rd)\times \calM^{p_2,q_2}_{v_{s_2,t_2}}(\rd)\longrightarrow M^{p,q}_{1\otimes v_s}(\rdd)
  \ee
  if and only if
  \be
   l_{ps_1}^{p_1/p}(\zd)\ast l_{ps_2}^{p_2/p}(\zd)\subset  l_{ps}^{q/p},\ \ \ l_{pt_1}^{q_1/p}\ast l_{pt_2}^{q_2/p}\subset  l_{ps}^{q/p}.
  \ee
  By Lemma \ref{lm-exp-wcov}, we find that the above two convolution inequalities are equivalent to
  \be
  (q/p,p_1/p,p_2/p,ps,ps_1,ps_2),\ (q/p,q_1/p,q_2/p,ps,pt_1,pt_2)\in \calA.
  \ee
\end{proof}

\begin{theorem}\label{thm-W4}
Suppose that $p<\fy$ and $p\leq p_i,q_i,q$ for $i=1,2$, $\tau\in (0,1)$. We have
\be
W_{\tau}:  \calM^{p_1,q_1}_{v_{s_1,t_1}}\times \calM^{p_2,q_2}_{v_{s_2,t_2}}\longrightarrow W(\scrF L^p, L^q_m)
\ee
if and only if
\be
(q/p,p_1/p,p_2/p,ps,ps_1,ps_2),\ (q/p,q_1/p,q_2/p,ps,pt_1,pt_2)\in \calA.
\ee
\end{theorem}
\begin{proof}
  As the proof of Theorem \ref{thm-M4},
  this is a direct conclusion of Theorem \ref{thm-W3} and Lemma \ref{lm-exp-wcov}.
\end{proof}

Next, using Lemma \ref{lm-exp-cov}, we recapture the main results of BMM in \cite{CorderoNicola2018IMRNI,Cordero2020a}.

\begin{theorem}[see also \cite{CorderoNicola2018IMRNI,Cordero2020a}]\label{thm-M5}
Let $0<p, q, p_i, q_i\leq \fy$ for $i=1,2$, $\tau\in [0,1]$. We have
\be
W_{\tau}: \calM^{p_1,q_1}\times \calM^{p_2,q_2}\longrightarrow M^{p,q}
\ee
if and only if
\ben\label{thm-M5-c1}
p_i,q_i\leq q,\ \ \ i=1,2
\een
and
\ben\label{thm-M5-c2}
\frac{1}{p_1}+\frac{1}{p_2}\geq  \frac{1}{p}+\frac{1}{q},\ \
\frac{1}{q_1}+\frac{1}{q_2}\geq  \frac{1}{p}+\frac{1}{q}.
\een
Moreover, if \eqref{thm-M5-c1} and \eqref{thm-M5-c2} hold, we have
\be
W_{\tau}: \calM^{p_1,q_1}_{|s|}\times \calM^{p_2,q_2}_s\longrightarrow M^{p,q}_s.
\ee
\end{theorem}
\begin{proof}
  The case $p<\fy$ follows by Theorem \ref{thm-M3} and Lemma \ref{lm-exp-cov}.
  For $p=\fy$, the corresponding results can be verified by Theorem \ref{thm-M3} and Lemma \ref{lm-exp-eb}.
\end{proof}

Similarly, we also have following result for BMW.
\begin{theorem}\label{thm-W5}
Let $0<p, q, p_i, q_i\leq \fy$ for $i=1,2$, $\tau\in (0,1)$. We have
\be
W_{\tau}: \calM^{p_1,q_1}\times \calM^{p_2,q_2}\longrightarrow W(\scrF L^p, L^q)
\ee
if and only if
\ben\label{thm-W5-c1}
p_i,q_i\leq q,\ \ \ i=1,2
\een
and
\ben\label{thm-W5-c2}
\frac{1}{p_1}+\frac{1}{p_2}\geq  \frac{1}{p}+\frac{1}{q},\ \
\frac{1}{q_1}+\frac{1}{q_2}\geq  \frac{1}{p}+\frac{1}{q}.
\een
Moreover, if \eqref{thm-W5-c1} and \eqref{thm-W5-c2} hold, we have
\be
W_{\tau}: \calM^{p_1,q_1}_{|s|}\times \calM^{p_2,q_2}_s\longrightarrow W(\scrF L^p, L^q_s).
\ee
\end{theorem}

\subsection{Sharp exponents for embedding relations}
In order to get the sharp exponents of embedding relations mentioned in Theorems \ref{thm-M3} and \ref{thm-W3},
we recall the following lemma.

\begin{lemma}[Sharpness of embedding, discrete form] \label{lm-exp-eb}
Suppose $0<q,q_1,q_2\leq \infty$, $s,s_1,s_2\in \mathbb{R}$. Then
\begin{equation*}
l^{q_1}_{v_{s_1}}(\rd)\subset l^{q_2}_{v_{s_2}}(\rd)
\end{equation*}
holds if and only if
\begin{equation*}
\begin{cases}
s_2\leq s_1  \\
\frac{1}{q_2}+\frac{s_2}{d}< \frac{1}{q_1}+\frac{s_1}{d}
\end{cases}
\text{or} \hspace{10mm}
\begin{cases}
s_2=s_1\\
q_2=q_1.
\end{cases}
\end{equation*}
\end{lemma}
Using this lemma, we conclude the sharp exponent characterizations for BMM and BMW.

\begin{theorem}\label{thm-M6}
Suppose $p\geq q$, $\tau\in [0,1]$. We have
\be
W_{\tau}:  \calM^{p_1,q_1}_{v_{s_1,t_1}}(\rd)\times \calM^{p_2,q_2}_{v_{s_2,t_2}}(\rd)
\longrightarrow M^{p,q}_{1\otimes v_s}(\rdd)
\ee
if and only if
\be
\begin{cases}
s\leq s_1,s_2,t_1,t_2,\\
\frac{1}{q}+\frac{s}{d}< \frac{1}{p_1}+\frac{s_1}{d}\ \ \text{or}\ \  (q,s)=(p_1,s_1),\\
\frac{1}{q}+\frac{s}{d}< \frac{1}{p_2}+\frac{s_2}{d}\ \ \text{or}\ \  (q,s)=(p_2,s_2),\\
\frac{1}{q}+\frac{s}{d}< \frac{1}{q_1}+\frac{t_1}{d}\ \ \text{or}\ \  (q,s)=(q_1,t_1),\\
\frac{1}{q}+\frac{s}{d}< \frac{1}{q_2}+\frac{t_2}{d}\ \ \text{or}\ \  (q,s)=(q_2,t_2).\\
\end{cases}
\ee
\end{theorem}
\begin{proof}
  By Theorem \ref{thm-M3}, we have
  \be
  W_{\tau}:  \calM^{p_1,q_1}_{v_{s_1,t_1}}(\rd)\times \calM^{p_2,q_2}_{v_{s_2,t_2}}(\rd)\longrightarrow M^{p,q}_{1\otimes v_s}(\rdd)
  \ee
  if and only if
  \be
   l_{s_1}^{p_1}(\zd),\  l_{s_2}^{p_2}(\zd)\subset  l_{s}^{q}(\zd),\ \ \
 \ \ \ l_{t_1}^{q_1}(\zd),\  l_{t_2}^{q_2}(\zd)\subset  l_{s}^{q}(\zd).
  \ee
  Then, the desired conclusion follows by Lemma \ref{lm-exp-eb}.
\end{proof}

\begin{theorem}\label{thm-W6}
Suppose $p\geq q$, $\tau\in (0,1)$. We have
\be
W_{\tau}:  \calM^{p_1,q_1}_{v_{s_1,t_1}}\times \calM^{p_2,q_2}_{v_{s_2,t_2}}\longrightarrow W(\scrF L^p, L^q_m)
\ee
if and only if
\be
\begin{cases}
s\leq s_1,s_2,t_1,t_2,\\
\frac{1}{q}+\frac{s}{d}< \frac{1}{p_1}+\frac{s_1}{d}\ \ \text{or}\ \  (q,s)=(p_1,s_1),\\
\frac{1}{q}+\frac{s}{d}< \frac{1}{p_2}+\frac{s_2}{d}\ \ \text{or}\ \  (q,s)=(p_2,s_2),\\
\frac{1}{q}+\frac{s}{d}< \frac{1}{q_1}+\frac{t_1}{d}\ \ \text{or}\ \  (q,s)=(q_1,t_1),\\
\frac{1}{q}+\frac{s}{d}< \frac{1}{q_2}+\frac{t_2}{d}\ \ \text{or}\ \  (q,s)=(q_2,t_2).\\
\end{cases}
\ee
\end{theorem}
\begin{proof}
  As the proof of Theorem \ref{thm-M4},
  this is a direct conclusion of Theorem \ref{thm-W3} and Lemma \ref{lm-exp-eb}.
\end{proof}

Following proposition is prepared for further separation of the endpoint cases
of BMW with power weights, i.e., for the proof of Theorem \ref{thm-W3} with $\tau=0,1$.

\begin{proposition}\label{pp-sep2-eb}
  Assume $p_1, q_1, p, q \in (0,\fy]$, $s,t\in \rr$. Then
  \ben\label{pp-sep2-eb-c0}
  l^{p_1,q_1}\subset l^{(q,p)}_{v_{s,t}}
  \een
  holds if and only if the following three embedding relations:
  \ben\label{pp-sep2-eb-c1}
  l^{p_1}\subset l^q_{v_s},\ \ \ l^{q_1}\subset l^p_{v_t},\ \ \ l^{q_1}\subset l^q_{v_{s+t}}.
  \een
\end{proposition}
\begin{proof}
  We divide this proof into two parts.

  \textbf{``Only if'' part.} Write \eqref{pp-sep2-eb-c0} by
  \ben\label{pp-sep2-eb-1}
  \bigg(\sum_{k\in \zd}\bigg(\sum_{n\in \zd}|a_{k,n}|^p\langle n\rangle^{tp}\bigg)^{q/p}\langle k\rangle^{sq}\bigg)^{1/q}
  \lesssim
  \bigg(\sum_{n\in \zd}\bigg(\sum_{k\in \zd}|a_{k,n}|^{p_1}\bigg)^{q_1/p_1}\bigg)^{1/q_1}.
  \een
  Then, we text the above inequality by several constructed sequences.

  \textbf{Test 1.}
  Set $a_{k,n}=0$ if $n\neq 0_{\zd}$. The inequality \eqref{pp-sep2-eb-1} becomes
  \be
  \bigg(\sum_{k\in \zd}|a_{k,0}|^q\langle k\rangle^{sq}\bigg)^{1/q}
  \lesssim
  \bigg(\sum_{k\in \zd}|a_{k,0}|^{p_1}\bigg)^{1/p_1}.
  \ee
  This implies the embedding relations $l^{p_1}\subset l^q_{v_s}$.

  \textbf{Test 2.}
  Set $a_{k,n}=0$ if $k\neq 0_{\zd}$. The inequality \eqref{pp-sep2-eb-1} becomes
  \be
  \bigg(\sum_{n\in \zd}|a_{0,n}|^p\langle n\rangle^{tp}\bigg)^{1/p}
  \lesssim
  \bigg(\sum_{n\in \zd}|a_{0,n}|^{q_1}\bigg)^{1/q_1}.
  \ee
  We obtain the embedding relation $l^{q_1}\subset l^p_{v_t}$.

  \textbf{Test 3.}
  Set $a_{k,n}=0$ if $k\neq n$. From the inequality \eqref{pp-sep2-eb-1} we have
  \be
  \bigg(\sum_{k\in \zd}|a_{k,k}|^q\langle k\rangle^{(s+t)q}\bigg)^{1/q}
  \lesssim
  \bigg(\sum_{k\in \zd}|a_{k,k}|^{q_1}\bigg)^{1/q_1}.
  \ee
  This is just the embedding relation $l^{q_1}\subset l^q_{v_{s+t}}$.

  \textbf{``Only if'' part.}
  Applying Lemma \ref{lm-exp-eb} to the embedding relations $l^{p_1}\subset l^q_{v_s}$ and $l^{q_1}\subset l^p_{v_t}$,
  we obtain $s\leq 0$ and $t\leq 0$.
  From this, we consider following three cases.

  \textbf{Case 1: $s=0$.}
  In this case, the embedding relations \eqref{pp-sep2-eb-c1} can be written as
  \be
  l^{p_1}\subset l^q,\ \ \ l^{q_1}\subset l^p_{v_t},\ \ \ l^{q_1}\subset l^q_{v_t}.
  \ee
  Correspondingly, our target is to verify
    \be
  l^{p_1,q_1}\subset l^{(q,p)}_{v_{0,t}}.
  \ee

  If $q/p\geq 1$, the Minkowski inequality implies that
  \be
  \bigg(\sum_{k\in \zd}\bigg(\sum_{n\in \zd}|a_{k,n}|^p\langle n\rangle^{tp}\bigg)^{q/p}\bigg)^{1/q}
  \lesssim
  \bigg(\sum_{n\in \zd}\bigg(\sum_{k\in \zd}|a_{k,n}|^{q}\bigg)^{p/q}\langle n\rangle^{tp}\bigg)^{1/p}.
  \ee
  Then, we use the embedding relations $l^{p_1}\subset l^q, l^{q_1}\subset l^p_{v_t}$ to deduce that
  \be
  \bigg(\sum_{n\in \zd}\bigg(\sum_{k\in \zd}|a_{k,n}|^{q}\bigg)^{p/q}\langle n\rangle^{tp}\bigg)^{1/p}
  \lesssim
  \bigg(\sum_{n\in \zd}\bigg(\sum_{k\in \zd}|a_{k,n}|^{p_1}\bigg)^{q_1/p_1}\bigg)^{1/q_1}.
  \ee
  The desired conclusion follows by the above two estimates.

  If $q/p< 1$, we use the embedding relations $l^q\subset l^p$ to deduce that
   \be
  \bigg(\sum_{k\in \zd}\bigg(\sum_{n\in \zd}|a_{k,n}|^p\langle n\rangle^{tp}\bigg)^{q/p}\bigg)^{1/q}
  \lesssim
  \bigg(\sum_{k\in \zd}\sum_{n\in \zd}|a_{k,n}|^{q}\langle n\rangle^{tq}\bigg)^{1/q}.
  \ee
  Then, the embedding relations $l^{p_1}\subset l^q, l^{q_1}\subset l^q_{v_t}$ further implies that
 \be
  \bigg(\sum_{k\in \zd}\sum_{n\in \zd}|a_{k,n}|^{q}\langle n\rangle^{tq}\bigg)^{1/q}
  \lesssim
  \bigg(\sum_{n\in \zd}\bigg(\sum_{k\in \zd}|a_{k,n}|^{p_1}\bigg)^{q_1/p_1}\bigg)^{1/q_1}.
  \ee
  The desired conclusion follows by the above two estimates.

  \textbf{Case 2: $t=0$.}
  In this case, the embedding relations \eqref{pp-sep2-eb-c1} can be written as
  \be
  l^{p_1}\subset l^q_{v_s},\ \ \ l^{q_1}\subset l^p,\ \ \ l^{q_1}\subset l^q_{v_s}.
  \ee
  Correspondingly, our target is to verify
    \be
  l^{p_1,q_1}\subset l^{(q,p)}_{v_{s,0}}.
  \ee

  If $p_1/q_1\geq 1$, we use the embedding relations $l^{p_1}\subset l^q_{v_s}, l^{q_1}\subset l^p$ to deduce that
   \be
  \bigg(\sum_{k\in \zd}\bigg(\sum_{n\in \zd}|a_{k,n}|^p\bigg)^{q/p}\langle k\rangle^{sq}\bigg)^{1/q}
  \lesssim
  \bigg(\sum_{k\in \zd}\bigg(\sum_{n\in \zd}|a_{k,n}|^{q_1}\bigg)^{p_1/q_1}\bigg)^{1/p_1}.
  \ee
  Then, we use the Minkowski inequality to continue this estimate by
  \be
  \bigg(\sum_{k\in \zd}\bigg(\sum_{n\in \zd}|a_{k,n}|^{q_1}\bigg)^{p_1/q_1}\bigg)^{1/p_1}
  \lesssim
  \bigg(\sum_{n\in \zd}\bigg(\sum_{k\in \zd}|a_{k,n}|^{p_1}\bigg)^{q_1/p_1}\bigg)^{1/q_1}.
  \ee
  The desired conclusion follows by the above two estimates.

  If $p_1/q_1< 1$, we use the embedding relations $l^{q_1}\subset l^q_{v_s}, l^{q_1}\subset l^p$ to deduce that
   \be
  \bigg(\sum_{k\in \zd}\bigg(\sum_{n\in \zd}|a_{k,n}|^p\bigg)^{q/p}\langle k\rangle^{sq}\bigg)^{1/q}
  \lesssim
  \bigg(\sum_{k\in \zd}\sum_{n\in \zd}|a_{k,n}|^{q_1}\bigg)^{1/q_1}.
  \ee
  Then, the well known embedding $l^{p_1}\subset \ l^{q_1}$ yields that
 \be
  \bigg(\sum_{k\in \zd}\sum_{n\in \zd}|a_{k,n}|^{q_1}\bigg)^{1/q_1}
  \lesssim
  \bigg(\sum_{n\in \zd}\bigg(\sum_{k\in \zd}|a_{k,n}|^{p_1}\bigg)^{q_1/p_1}\bigg)^{1/q_1}.
  \ee
  The desired conclusion follows by the above two estimates.

  \textbf{Case 3: $s, t<0$.} In this case, by Lemma \ref{lm-exp-eb} and the embedding relations \eqref{pp-sep2-eb-c1}, we obtain that
  \be
  \frac{1}{q}+\frac{s}{d}<\frac{1}{p_1},\ \ \ \frac{1}{p}<\frac{1}{q_1}-\frac{t}{d},\ \ \  \frac{1}{q}+\frac{s}{d}<\frac{1}{q_1}-\frac{t}{d}.
  \ee
  From this, there exists a sufficiently small constant $\ep>0$ such that
  \be
    \frac{1}{q}+\frac{s}{d}+\ep<\frac{1}{p_1},\ \ \ \frac{1}{p}<\frac{1}{q_1}-\frac{t}{d}-\ep, \ \ \  \frac{1}{q}+\frac{s}{d}+\ep<\frac{1}{q_1}-\frac{t}{d}-\ep.
  \ee
  Set
  \be
  \frac{1}{\r}:=\max\{\frac{1}{q}+\frac{s}{d}+\ep, 0\},\ \ \ \    \frac{1}{r}:=\frac{1}{q_1}-\frac{t}{d}-\ep.
  \ee
  We have $\frac{1}{\r}<\frac{1}{r}$ and the following relations:
  \be
  \frac{1}{\r}+\frac{s}{d}<\frac{1}{\r}<\frac{1}{p_1},
  \ \ \frac{1}{p}+\frac{t}{d}<\frac{1}{r}+\frac{t}{d}<\frac{1}{q_1}.
  \ee
  From this and Lemma \ref{lm-exp-eb}, we obtain the following embedding relations
  \be
  l^{p_1}\subset l^{\r}\subset \ l^q_{v_s},\ \ \ l^{q_1}\subset l^r_{v_{t}}\subset l^p_{v_t}.
  \ee
  Using these embedding relations and the Minkowski inequality with $\frac{1}{\r}<\frac{1}{r}$, we deduce that
  \be
  \begin{split}
  \bigg(\sum_{k\in \zd}\bigg(\sum_{n\in \zd}|a_{k,n}|^p\langle n\rangle^{tp}\bigg)^{q/p}\langle k\rangle^{sq}\bigg)^{1/q}
  \lesssim &
  \bigg(\sum_{k\in \zd}\bigg(\sum_{n\in \zd}|a_{k,n}|^r\langle n\rangle^{tr}\bigg)^{\r/r}\bigg)^{1/\r}
  \\
  \lesssim
  \bigg(\sum_{n\in \zd}\bigg(\sum_{k\in \zd}|a_{k,n}|^{\r}\bigg)^{r/\r}&\langle n\rangle^{tr}\bigg)^{1/r}
  \lesssim
  \bigg(\sum_{n\in \zd}\bigg(\sum_{k\in \zd}|a_{k,n}|^{p_1}\bigg)^{q_1/p_1}\bigg)^{1/q_1}.
  \end{split}
  \ee
  This is the desired conclusion.
\end{proof}

\begin{proof}[Proof of Theorem \ref{thm-W3}]
The case $\tau\in (0,1)$ follows directly by Theorem \ref{thm-W2}.
For $\tau=0$, observe that
\be
l^{p_1,q_1}_{v_{s_1,t_1}}\subset l^{(q,p)}_{v_s\otimes 1}\Longleftrightarrow l^{p_1,q_1}\subset l^{(q,p)}_{v_{s-s_1,-t_1}}.
\ee
Then, Proposition \ref{pp-sep2-eb} tells us
\be
l^{p_1,q_1}\subset l^{(q,p)}_{v_{s-s_1,-t_1}}\Longleftrightarrow
l^{p_1}\subset l^q_{s-s_1},\  l^{q_1}\subset l^p_{-t_1},\  l^{q_1}\subset l^q_{s-s_1-t_1},
\ee
which is equivalent to
\be
l^{p_1}_{s_1}\subset l^q_{s},\  l^{q_1}_{t_1}\subset l^p,\  l^{q_1}_{s_1+t_1}\subset l^q_{s}.
\ee
Hence, we have
\be
l^{p_1,q_1}_{v_{s_1,t_1}}\subset l^{(q,p)}_{v_s\otimes 1}\Longleftrightarrow l^{p_1}_{s_1}\subset l^q_{s},\  l^{q_1}_{t_1}\subset l^p,\  l^{q_1}_{s_1+t_1}\subset l^q_{s}.
\ee
Combining this with Theorem \ref{thm-W2}, we obtain the desired conclusion.
The case $\tau=1$ follows by the same argument of $\tau=0$, we omit the detail.
\end{proof}

Using Lemma \ref{lm-exp-eb} and Theorem \ref{thm-W3}, we obtain the following exponents characterization for BMW with endpoints.
\begin{theorem}\label{thm-W7}
Assume $p_i, q_i, p, q \in (0,\fy]$, $i=1,2$, $\tau=0,1$. We have
\be
W_{\tau}: \calM^{p_1,q_1}_{v_{s_1,t_1}}(\rd)\times \calM^{p_2,q_2}_{v_{s_2,t_2}}(\rd)
\longrightarrow W(\scrF L^p, L^q_s)(\rdd)
\ee
if and only if

\be
(\tau=0)\begin{cases}
s\leq s_1,t_2,\ \ 0\leq t_1,s_2,\\
\frac{1}{p}\leq \frac{1}{q_1}+\frac{t_1}{d}\ \ \text{or}\ \  (q_1,t_1)=(p,0),\\
\frac{1}{p}\leq \frac{1}{p_2}+\frac{s_2}{d}\ \ \text{or}\ \  (p_2,s_2)=(p,0),\\
\frac{1}{q}+\frac{s}{d}\leq \frac{1}{p_1}+\frac{s_1}{d}\ \ \text{or}\ \  (q,s)=(p_1,s_1),\\
\frac{1}{q}+\frac{s}{d}\leq \frac{1}{q_1}+\frac{s_1+t_1}{d}\ \ \text{or}\ \  (q,s)=(q_1,s_1+t_1),\\
\frac{1}{q}+\frac{s}{d}\leq \frac{1}{q_2}+\frac{t_2}{d}\ \ \text{or}\ \  (q,s)=(q_2,t_2),\\
\end{cases}
\ee

and

\be
(\tau=1)\begin{cases}
s\leq s_2,t_1,\ \ 0\leq t_2,s_1,\\
\frac{1}{p}\leq \frac{1}{q_2}+\frac{t_2}{d}\ \ \text{or}\ \  (q_2,t_2)=(p,0),\\
\frac{1}{p}\leq \frac{1}{p_1}+\frac{s_1}{d}\ \ \text{or}\ \  (p_1,s_1)=(p,0),\\
\frac{1}{q}+\frac{s}{d}\leq \frac{1}{p_2}+\frac{s_2}{d}\ \ \text{or}\ \  (q,s)=(p_2,s_2),\\
\frac{1}{q}+\frac{s}{d}\leq \frac{1}{q_2}+\frac{s_2+t_2}{d}\ \ \text{or}\ \  (q,s)=(q_2,s_2+t_2),\\
\frac{1}{q}+\frac{s}{d}\leq \frac{1}{q_1}+\frac{t_1}{d}\ \ \text{or}\ \  (q,s)=(q_1,t_1).\\
\end{cases}
\ee

\end{theorem}

\section{Complements: pesudodifferential operators on modulation spaces}
\subsection{Relations between BMM (BMW) and BPM (BPW)}
In order to give characterizations of BPM and BPW, we would like to first establish
some equivalent relations associated with BMM and BMW. We point out that these equivalent
relations follows by the dual arguments of function spaces, which have been wildly used before,
for instance, one can see the proof of \cite[Theorem 5.1]{CorderoNicola2018IMRNI} for this direction.

\begin{proposition}\label{pp-eqOP-BMM}
  Assume $1\leq p, q, p_i, q_i \leq \fy$, $i=1,2$, $\tau\in [0,1]$. Then the following statements are equivalent:
  \begin{eqnarray*}
  (i)& &\forall\sigma\in M^{p,q}_{1\otimes m}(\rdd)
  \Longrightarrow OP_{\tau}(\sigma): \calM^{p_1,q_1}_{m_1}(\rd)\longrightarrow M^{p_2,q_2}_{m_2}(\rd),
  \\
  (ii)&  &\|OP_{\tau}(\s)f\|_{M^{p_2,q_2}_{m_2}(\rd)}
  \lesssim \|\s\|_{M^{p,q}_{1\otimes m}(\rdd)}\|f\|_{M^{p_1,q_1}_{m_1}(\rd)},\ \ \ f\in \calS(\rd),
  \\
  (iii)&   &W_{\tau}: \calM^{p_2',q_2'}_{m_2^{-1}}(\rd)\times \calM^{p_1,q_1}_{m_1}(\rd)
  \longrightarrow M^{p',q'}_{1\otimes m^{-1}}(\rdd).
  \end{eqnarray*}
\end{proposition}
\begin{proof}
  $(ii)\Longrightarrow (i)$ is clear. In order to verify $(i)\Longrightarrow (ii)$,
  we apply the Closed Graph Theorem as in \cite[Proposition 4.7]{CorderoNicola2010IMRNI}.
  The map acting from $M^{p,q}_{1\otimes m}$ into $B(\calM^{p_1,q_1}_{m_1}, M^{p_2,q_2}_{m_2})$ is defined by
  \be
  P: \sigma \longrightarrow OP_{\tau}(\s).
  \ee
  Given any sequence of pairs that $(\s_n,OP_{\tau}(\s_n))$ tends to $(\s,T)$ in the topology of graph of $P$,
  for any $f,g\in \calS(\rd)$ we have
  \be
  \begin{split}
  \langle OP_{\tau}(\s)f,g\rangle
  = &
  \langle \s, W_{\tau}(g,f)\rangle
  \\
  = &
  \lim_{n\rightarrow \fy}\langle \s_n, W_{\tau}(g,f)\rangle
  =
  \lim_{n\rightarrow \fy}\langle OP_{\tau}(\s_n)f,g\rangle=\langle Tf,g\rangle.
  \end{split}
  \ee
  From this, we obtain $T=OP_{\tau}(\s)$. Then the graph of $P$ is closed. Hence, $P$ is bounded,
  i.e., (ii) is valid.

  Next, we turn to the proof of $(ii)\Longleftrightarrow (iii)$. This follows by a standard dual argument.
  If (ii) holds, for any $f,g\in \calS(\rd)$ and $\s\in M^{p,q}_{1\otimes m}$, we have
  \be
  \begin{split}
  |\langle W_{\tau}(f,g),\s\rangle|
  = &
  |\langle f,OP_{\tau}(\s)g\rangle|
  \\
  \lesssim &
  \|f\|_{M^{p_2',q_2'}_{m_2^{-1}}}\|OP_{\tau}(\s)g\|_{M^{p_2,q_2}_{m_2}}
  \lesssim
  \|f\|_{M^{p_2',q_2'}_{m_2^{-1}}}\|\s\|_{M^{p,q}_{1\otimes m}}\|g\|_{M^{p_1,q_1}_{m_1}}.
  \end{split}
  \ee
  Then, the duality of modulation spaces implies (iii).

  Viceversa, if (iii) holds, for any $f,g\in \calS(\rd)$ and $\s\in M^{p,q}_{1\otimes m}$, we have
  \be
  \begin{split}
  |\langle OP_{\tau}(\s)f,g\rangle|
  = &
  |\langle \s, W_{\tau}(g,f)\rangle|
  \\
  \lesssim &
  \|\s\|_{M^{p,q}_{1\otimes m}}\|W_{\tau}(g,f)\|_{M^{p',q'}_{1\otimes m^{-1}}}
  \lesssim
  \|\s\|_{M^{p,q}_{1\otimes m}}\|g\|_{M^{p_2',q_2'}_{m_2^{-1}}}\|f\|_{M^{p_1,q_1}_{m_1}}.
  \end{split}
  \ee
  Then, (iii) follows by the duality of modulation spaces.
\end{proof}

By a similar argument, we give following equivalent relations between BMW and BPW.

\begin{proposition}\label{pp-eqOP-BMW}
  Assume $1\leq p, q, p_i, q_i \leq \fy$, $i=1,2$, $\tau\in [0,1]$. Then the following statements are equivalent:
  \begin{eqnarray*}
  (i)& &\forall\sigma\in \scrF M^{p,q}_{1\otimes m}(\rdd)
  \Longrightarrow OP_{\tau}(\sigma): \calM^{p_1,q_1}_{m_1}(\rd)\longrightarrow M^{p_2,q_2}_{m_2}(\rd),
  \\
  (ii)&  &\|OP_{\tau}(\s)f\|_{M^{p_2,q_2}_{m_2}(\rd)}
  \lesssim \|\s\|_{\scrF M^{p,q}_{1\otimes m}(\rdd)}\|f\|_{M^{p_1,q_1}_{m_1}(\rd)},\ \ \ f\in \calS(\rd),
  \\
  (iii)&   &W_{\tau}: \calM^{p_2',q_2'}_{m_2^{-1}}(\rd)\times \calM^{p_1,q_1}_{m_1}(\rd)
  \longrightarrow \scrF M^{p',q'}_{1\otimes m^{-1}}(\rdd).
  \end{eqnarray*}
\end{proposition}

\subsection{Sharp exponents of BPM and BPW}

By propositions \ref{pp-eqOP-BMM} and \ref{pp-eqOP-BMW},
the estimates of the $\tau$-Wigner distribution can be translated in ones into the
corresponding results for $\tau$-operators. Following, we collect some important cases, the interested reader can deduce
the results they need.

Following result is a direct conclusion by
 Proposition \ref{pp-eqOP-BMM} and Theorem \ref{thm-M5}.

\begin{theorem}\label{thm-OPM}
  Assume $1\leq p, q, p_i, q_i \leq \fy$, $i=1,2$, $\tau\in [0,1]$. We have
  \be
  \forall\sigma\in M^{p,q}(\rdd)
  \Longrightarrow OP_{\tau}(\sigma): \calM^{p_1,q_1}(\rd)\longrightarrow M^{p_2,q_2}(\rd)
  \ee
  if and only if
\be
p_1,q_1,p_2',q_2'\leq q',\ \ \ i=1,2
\ee
and
\be
\frac{1}{p_1}+\frac{1}{p_2'}\geq  \frac{1}{p'}+\frac{1}{q'},\ \
\frac{1}{q_1}+\frac{1}{q_2'}\geq  \frac{1}{p'}+\frac{1}{q'}.
\ee
\end{theorem}

Next, we want to establish the sharp exponents for the boundedness on modulation spaces with power weights
of pseudodifferential operators with symbols in Sj\"{o}strand's class.
Before this, we first give following characterization of BPM, which can be directly deduced by
Proposition \ref{pp-eqOP-BMM} and Theorem \ref{thm-M3}.

\begin{theorem}\label{thm-OPM-PW}
  Assume $1\leq p, q, p_i, q_i \leq \fy$, $i=1,2$, $\tau\in [0,1]$. We have
  \be
  \forall\sigma\in M^{p,q}_{1\otimes v_s}(\rdd)
  \Longrightarrow OP_{\tau}(\sigma): \calM^{p_1,q_1}_{v_{s_1,t_1}}(\rd)\longrightarrow M^{p_2,q_2}_{v_{s_2,t_2}}(\rd)
  \ee
  if and only if
\begin{eqnarray}
&l_{-p's_2}^{p_2'/p'}(\zd)\ast l_{p's_1}^{p_1/p'}(\zd)\subset  l_{-p's}^{q'/p'}(\zd),
\ \ l_{-p't_2}^{q_2'/p'}(\zd)\ast l_{p't_1}^{q_1/p'}(\zd)\subset  l_{-p's}^{q'/p'}(\zd),
&p>1,
\\
&l_{-s_2}^{p_2'}(\zd),\ l_{s_1}^{p_1}(\zd)\subset  l_{-s}^{q'}(\zd),
\ \ l_{-t_2}^{q_2'}(\zd),\ l_{t_1}^{q_1}(\zd)\subset  l_{-s}^{q'}(\zd),
&p\leq q.
\end{eqnarray}
\end{theorem}

We recall a special case of Lemma \ref{lm-exp-wcov} as follows.

\begin{lemma}\label{lm-exp-wcov-SJ}
Suppose $1\leq q_1, q_2\leq \infty$, $s_1, s_2\in \mathbb{R}$. Then
\begin{equation}
l^{q_1}_{s_1}(\rd)\ast l^{q_2}_{s_2}(\rd) \subset l^{\fy}(\rd)
\end{equation}
if and only if
\begin{eqnarray}
s_1=s_2=0,\  \frac{1}{q_1}+\frac{1}{q_2}=1
\ \ \ \text{or}\ \ \ \ \
\begin{cases}
0\leq s_1, s_2,\\
1<\frac{1}{q_1}+\frac{s_1}{d}+\frac{1}{q_2}+\frac{s_2}{d},\\
(q_1,s_1)=(\fy,0) ~\text{if}~ \frac{1}{q_1}+\frac{s_1}{d}=0, \\
(q_2,s_2)=(\fy,0) ~\text{if}~ \frac{1}{q_2}+\frac{s_2}{d}=0.
\end{cases}
\end{eqnarray}
\end{lemma}

Now, we are in a position to give the sharp exponents of BPM with Sj\"{o}strand's class.

\begin{theorem}\label{thm-OPM-SJ}
  Assume $1\leq p, q, p_i, q_i \leq \fy$, $i=1,2$, $\tau\in [0,1]$. We have
  \ben\label{thm-OPM-SJ-c1}
  \forall\sigma\in M^{\fy,1}(\rdd)\Longrightarrow
  OP_{\tau}(\sigma): \calM^{p_1,q_1}_{v_{s_1,t_1}}(\rd)\longrightarrow M^{p_2,q_2}_{v_{s_2,t_2}}(\rd)
  \een
  if and only if
\begin{eqnarray}\label{thm-OPM-SJ-c2}
s_1=s_2=0,\  p_1=p_2\ \ \ \ \ \    \text{or}
\begin{cases}
s_2\leq 0\leq s_1,\\
\frac{1}{p_2}+\frac{s_2}{d}<\frac{1}{p_1}+\frac{s_1}{d},\\
(p_1,s_1)=(\fy,0) ~\text{if}~ \frac{1}{p_1}+\frac{s_1}{d}=0, \\
(p_2,s_2)=(1,0) ~\text{if}~ \frac{1}{p_2}+\frac{s_2}{d}=1,
\end{cases}
\end{eqnarray}
and
\begin{eqnarray}\label{thm-OPM-SJ-c3}
t_1=t_2=0,\  q_1=q_2\ \ \ \ \ \    \text{or}
\begin{cases}
t_2\leq 0\leq t_1,\\
\frac{1}{q_2}+\frac{t_2}{d}<\frac{1}{q_1}+\frac{t_1}{d},\\
(q_1,t_1)=(\fy,0) ~\text{if}~ \frac{1}{q_1}+\frac{t_1}{d}=0, \\
(q_2,t_2)=(1,0) ~\text{if}~ \frac{1}{q_2}+\frac{t_2}{d}=1.
\end{cases}
\end{eqnarray}
In particular, when $s_i=t_i=0$, $i=1,2$, we have \eqref{thm-OPM-SJ-c1} holds if and only if
\ben\label{cp_SJ}
p_1\leq p_2,\ \ \ \ q_1\leq q_2.
\een
\end{theorem}
\begin{proof}
  Using Theorem \ref{thm-OPM-PW}, we obtain that \eqref{thm-OPM-SJ-c1} holds if and only if
\be
 l_{-s_2}^{p_2'}(\zd)\ast l_{s_1}^{p_1}(\zd)\subset  l^{\fy}(\zd),
 \ \ \ l_{-t_2}^{q_2'}(\zd)\ast l_{t_1}^{q_1}(\zd)\subset  l^{\fy}(\zd).
\ee
Then, Lemma \ref{lm-exp-wcov-SJ} tells us that the above two convolution inequalities are equivalent to
\eqref{thm-OPM-SJ-c2} and \eqref{thm-OPM-SJ-c3}.
\end{proof}

Observe that Wiener amalgam space $W(\scrF L^q, L^{p})$ has the same local regularity with $M^{p,q}$.
Moreover, for $p>q$ they have the following inclusion relations:
\be
M^{p,q}\subsetneq  W(\scrF L^q, L^{p}),\ \ \ M^{q,p}\supsetneq   W(\scrF L^p, L^{q}).
\ee
So, there is a natural question that for $p>q$ whether the boundedness of pseudodifferential operator with symbols in $M^{p,q}$ or
$W(\scrF L^p, L^{q})$ can be preserved with symbols in
$W(\scrF L^q, L^{p})$ or $M^{q,p}$ respectively.
With the help of our full characterizations of BMM and BMW, one can find that the answer is negative unless the trivial case happen, i.e., $p=q$. Here, we only give a detailed comparison for the Sj\"{o}strand's class $M^{\fy,1}$ and the corresponding larger space $W(\scrF L^1, L^{\fy})$. Let us being with the sharp exponents for the non-endpoint case of BPW.

\begin{theorem}\label{thm-OPW-SJ}
  Assume $1\leq p, q, p_i, q_i \leq \fy$, $i=1,2$, $\tau\in (0,1)$. We have
  \ben\label{thm-OPW-SJ-c1}
  \forall\sigma\in W(\scrF L^1, L^{\fy})(\rdd)\Longrightarrow
  OP_{\tau}(\sigma): \calM^{p_1,q_1}_{v_{s_1,t_1}}(\rd)\longrightarrow M^{p_2,q_2}_{v_{s_2,t_2}}(\rd)
  \een
  if and only if
\be
\begin{cases}
s_2\leq 0\leq s_1,\\
1<\frac{1}{p_1}+\frac{s_1}{d}\ \ \text{or}\ \ (p_1,s_1)=(1,0),\\
\frac{1}{p_2}+\frac{s_2}{d}<0\ \ \text{or}\ \ (p_2,s_2)=(\fy,0),
\end{cases}
\ee
and
\be
\begin{cases}
t_2\leq 0\leq t_1,\\
1<\frac{1}{q_1}+\frac{t_1}{d}\ \ \text{or}\ \ (q_1,t_1)=(1,0),\\
\frac{1}{q_2}+\frac{t_2}{d}<0\ \ \text{or}\ \ (q_2,t_2)=(\fy,0).
\end{cases}
\ee
In particular, when $s_i=t_i=0$, $i=1,2$, we have \eqref{thm-OPW-SJ-c1} holds if and only if
\ben\label{cp_SJW}
p_1=q_1=1,\ \ \ \ p_2=q_2=\fy.
\een
\end{theorem}
\begin{proof}
  By Proposition \ref{pp-eqOP-BMW} and Theorem \ref{thm-W3}, we conclude that \eqref{thm-OPW-SJ-c1} holds if and only if
  \be
  l_{-s_2}^{p_2'}(\zd),\ l_{s_1}^{p_1}(\zd)\subset  l^{1}(\zd),
\ \ l_{-t_2}^{q_2'}(\zd),\ l_{t_1}^{q_1}(\zd)\subset  l^{1}(\zd).
  \ee
  Then, the desired conclusion follows by Lemma \ref{lm-exp-eb}.
\end{proof}

\begin{remark}\label{rk_cpSJ}
Comparing \eqref{cp_SJW} with \eqref{cp_SJ}, we find that the range of exponents for BPW ($\tau\in (0,1)$) with symobls in $W(\scrF L^1, L^{\fy})(\rdd)$
is strictly small than that for BPM with Sj\"{o}strand's class.
\end{remark}

Next, we handle that endpoint case of BPW.
We first give following characterization for the endpoint cases of BPW, which can be directly deduced by
Proposition \ref{thm-W3} and Proposition \ref{pp-eqOP-BMW}.

\begin{theorem}\label{thm-OPWE-PW}
  Assume $1\leq p, q, p_i, q_i \leq \fy$, $i=1,2$, $\tau=0,1$. We have
  \be
  \forall\sigma\in W(\scrF L^p, L^q_s)(\rdd)
  \Longrightarrow OP_{\tau}(\sigma): \calM^{p_1,q_1}_{v_{s_1,t_1}}(\rd)\longrightarrow M^{p_2,q_2}_{v_{s_2,t_2}}(\rd)
  \ee
  if and only if
\begin{eqnarray}
&
l^{q_2'}_{-t_2}(\zd), l^{p_1}_{s_1}(\zd)\subset l^{p'}(\zd),\ \ \
l^{p_2'}_{-s_2}(\zd), l^{q_2'}_{-(s_2+t_2)}(\zd), l^{q_1}_{t_1}(\zd)\subset l^{q'}_{-s}(\zd),
&\tau=0,
\\
&
l^{q_1}_{t_1}(\zd),l^{p_2'}_{-s_2}(\zd)\subset l^{p'}(\zd),\ \ \
l^{p_1}_{s_1}(\zd), l^{q_1}_{s_1+t_1}(\zd), l^{q_2'}_{-t_2}(\zd)\subset l^{q'}_{-s}(\zd),
&\tau=1.
\end{eqnarray}
\end{theorem}

Then, corresponding to Theorem \ref{thm-W5}, we establish the sharp exponents for the endpoint cases of BPW with constant weights.

\begin{theorem}\label{thm-OPWE}
  Assume $1\leq p, q, p_i, q_i \leq \fy$, $i=1,2$, $\tau=0,1$. We have
  \ben\label{thm-OPWE-c1}
  \forall\sigma\in W(\scrF L^p, L^q)(\rdd)
  \Longrightarrow OP_{\tau}(\sigma): \calM^{p_1,q_1}(\rd)\longrightarrow M^{p_2,q_2}(\rd)
  \een
  if and only if
\begin{eqnarray}
&
q_2', p_1\leq p',\ \ \
p_2', q_2', q_1\leq q',
&\tau=0,
\\
&
q_1, p_2'\leq p',\ \ \
p_1, q_1, q_2'\leq q',
&\tau=1.
\end{eqnarray}
\end{theorem}
\begin{proof}
It follows by Theorem \ref{thm-OPWE-PW} that \eqref{thm-OPWE-c1} is equivalent to
  \begin{eqnarray*}
&
l^{q_2'}(\zd), l^{p_1}(\zd)\subset l^{p'}(\zd),\ \ \
l^{p_2'}(\zd), l^{q_2'}(\zd), l^{q_1}(\zd)\subset l^{q'}(\zd),
&\tau=0,
\\
&
l^{q_1}(\zd),l^{p_2'}(\zd)\subset l^{p'}(\zd),\ \ \
l^{p_1}(\zd), l^{q_1}(\zd), l^{q_2'}(\zd)\subset l^{q'}(\zd),
&\tau=1.
\end{eqnarray*}
Then, the desired conditions can be deduced by Lemma \ref{lm-exp-eb}.
\end{proof}

Now, we return to BPW with symbols in $W(\scrF L^{1}, L^{\fy})(\rdd)$.

\begin{theorem}\label{thm-OPWE-SJ}
  Assume $1\leq p, q, p_i, q_i \leq \fy$, $i=1,2$, $\tau=0,1$. We have
  \ben\label{thm-OPWE-SJ-c1}
  \forall\sigma\in W(\scrF L^{1}, L^{\fy})(\rdd)\Longrightarrow
  OP_{\tau}(\sigma): \calM^{p_1,q_1}_{v_{s_1,t_1}}(\rd)\longrightarrow M^{p_2,q_2}_{v_{s_2,t_2}}(\rd)
  \een
  if and only if
\be
(\tau=0)\begin{cases}
s_2,t_2\leq 0\leq s_1,t_1,\\
0< \frac{1}{p_1}+\frac{s_1}{d}\ \ \text{or}\ \  (p_1,s_1)=(\fy,0),\\
\frac{1}{p_2}+\frac{s_2}{d}<0\ \ \text{or}\ \ (p_2,s_2)=(\fy,0),\\
1< \frac{1}{q_1}+\frac{t_1}{d}\ \ \text{or}\ \  (q_1,t_1)=(1,0),\\
\frac{1}{q_2}+\frac{t_2}{d}<1\ \ \text{or}\ \ (q_2,t_2)=(1,0),\\
\frac{1}{q_2}+\frac{s_2+t_2}{d}<0\ \ \text{or}\ \  (q_2,s_2,t_2)=(\fy,0,0)\\
\end{cases}
\ee

and

\be
(\tau=1)\begin{cases}
s_2,t_2\leq 0\leq s_1,t_1,\\
1<\frac{1}{p_1}+\frac{s_1}{d}\ \ \text{or}\ \  (p_1,s_1)=(1,0),\\
\frac{1}{p_2}+\frac{s_2}{d}<1\ \ \text{or}\ \  (p_2,s_2)=(1,0),\\
0<\frac{1}{q_1}+\frac{t_1}{d}\ \ \text{or}\ \  (q_1,t_1)=(\fy,0),\\
\frac{1}{q_2}+\frac{t_2}{d}<0\ \ \text{or}\ \  (q_2,t_2)=(\fy,0),\\
1<\frac{1}{q_1}+\frac{s_1+t_1}{d}\ \ \text{or}\ \  (q_1,s_1,t_1)=(1,0,0)
\end{cases}.
\ee

In particular, when $s_i=t_i=0$, $i=1,2$, we have \eqref{thm-OPWE-SJ-c1} holds if and only if
\begin{eqnarray}
&\label{cp_SJW0}
q_1=1,\ p_2=q_2=\fy,\ \ \
&\tau=0,
\\
&\label{cp_SJW1}
q_2=\fy,\ p_1=q_1=1,\ \ \
&\tau=1.
\end{eqnarray}
\end{theorem}
\begin{proof}
  It follows from Theorem \ref{thm-OPWE-PW} that \eqref{thm-OPWE-SJ-c1} is equivalent to
\begin{eqnarray}
&
l^{q_2'}_{-t_2}(\zd), l^{p_1}_{s_1}(\zd)\subset l^{\fy}(\zd),\ \ \
l^{p_2'}_{-s_2}(\zd), l^{q_2'}_{-(s_2+t_2)}(\zd), l^{q_1}_{t_1}(\zd)\subset l^{1}(\zd),
&\tau=0,
\\
&
l^{q_1}_{t_1}(\zd),l^{p_2'}_{-s_2}(\zd)\subset l^{\fy}(\zd),\ \ \
l^{p_1}_{s_1}(\zd), l^{q_1}_{s_1+t_1}(\zd), l^{q_2'}_{-t_2}(\zd)\subset l^{1}(\zd).
&\tau=1.
\end{eqnarray}
Then, the desired conclusion follows from Lemma \ref{lm-exp-eb}.
\end{proof}

\begin{remark}\label{rk_cpSJE}
Comparing \eqref{cp_SJW0}, \eqref{cp_SJW1} with \eqref{cp_SJ}, one can find that the range of exponents for BPW ($\tau=0,1$) with symobls in $W(\scrF L^1, L^{\fy})(\rdd)$
is strictly small than that for BPM with Sj\"{o}strand's class.
\end{remark}

Finally, we consider the boundedness on Sobolev spaces $H^s$ of pseudodifferential operators
with symbols in Wiener amalgam spaces $W(\scrF L^p, L^q)(\rdd)$.
Note that $M^{2,2}_s=H^s$.

\begin{theorem}\label{thm-OPW-Hs}
  Assume $1\leq p, q, p_i, q_i \leq \fy$, $i=1,2$, $\tau=0,1$. We have
  \ben\label{thm-OPW-Hs-c1}
  \forall\sigma\in W(\scrF L^p, L^q)(\rdd)
  \Longrightarrow OP_{\tau}(\sigma): \calM^{2,2}_{v_{0,t_1}}(\rd)\longrightarrow M^{2,2}_{v_{0,t_2}}(\rd)
  \een
  if and only if
\be
t_2\leq 0\leq t_1,\ \ \  p,q\leq 2,
\ \ \ \  \tau=0,1.
\ee
\end{theorem}
\begin{proof}
  Recall that $M^{2,2}_s=H^s$, by Theorem \ref{thm-OPWE-PW}, we conclude that \eqref{thm-OPW-Hs-c1} is equivalent to
  \begin{eqnarray*}
&
l^{2}_{-t_2}(\zd), l^{2}(\zd)\subset l^{p'}(\zd),\ \ \
l^{2}(\zd), l^{2}_{-t_2}(\zd), l^{2}_{t_1}(\zd)\subset l^{q'}(\zd),
&\tau=0,
\\
&
l^{2}_{t_1}(\zd),l^{2}(\zd)\subset l^{p'}(\zd),\ \ \
l^{2}(\zd), l^{2}_{t_1}(\zd), l^{2}_{-t_2}(\zd)\subset l^{q'}(\zd),
&\tau=1,
\end{eqnarray*}
which implies the desired conclusion by Lemma \ref{lm-exp-eb}.
\end{proof}
\begin{remark}
  From Theorem \ref{thm-OPW-Hs}, for $\tau=0,1$, we observe that for any $s\geq 0$, there exists
  a symbol $\s\in W(\scrF L^{\fy}, L^1)(\rdd)$ such that
  the corresponding pseudodifferential operators $OP_{\tau}(\sigma)$
  are not bounded from $L^2(\rd)$ to $H^s(\rd)$.
\end{remark}

\subsection*{Acknowledgements}
This work was partially supported by the National Natural Science Foundation of China (Nos. 11701112, 11601456, 11671414, 11771388).

\bibliographystyle{abbrv}

\end{document}